\documentclass[preprint,12pt]{elsarticleHF}
\usepackage{amssymb}
\usepackage{amsmath, amsthm}
\newtheorem{theorem}{Theorem}
\newtheorem{definition}{Definition}
\newtheorem{lemma}{Lemma}
\newproof{pf}{Proof}
\newtheorem{remark}{Remark}
\newtheorem{corollary}{Corollary}

\newtheorem{example}{Example}

\usepackage{graphicx}
\usepackage{stmaryrd}
\usepackage{arydshln}
\usepackage[dvipsnames]{xcolor}
\usepackage{tikz}
\usepackage{latexsym, amsmath}
\usetikzlibrary{shapes,snakes}
\usetikzlibrary{matrix}
\usetikzlibrary{intersections}
\usetikzlibrary{calc}
\usetikzlibrary{positioning}
\usetikzlibrary{arrows}
\usetikzlibrary{fit}
\usepackage{textpos}
\usepackage{mathdots}
\usepackage{enumitem}
\setitemize{leftmargin=*}
\usepackage[autostyle]{csquotes}

\journal{LAA}

\begin{document}

\begin{frontmatter}



\title{On a new kind of Ansatz Spaces for Matrix Polynomials}


\author[label1]{Heike Fa\ss bender}
\address[label1]{Institut \emph{Computational Mathematics}/ AG Numerik, TU Braunschweig, Pockelsstr. 14, 38106 Braunschweig, Germany}
\cortext[cor1]{Corresponding author, Email philip.saltenberger@tu-braunschweig.de}
\author[label1]{Philip Saltenberger\corref{cor1}}

\begin{abstract}
In this paper, we introduce a new family of equations for matrix pencils that may be utilized for the construction
of strong
linearizations for any square or rectangular matrix polynomial. We provide a comprehensive characterization of the
resulting
vector spaces and show that almost every matrix pencil therein
is a strong linearization regardless whether the matrix polynomial under consideration is regular
or singular. These novel ``ansatz spaces'' cover all block Kronecker pencils as introduced in
\cite{DopLPVD16} as a subset and therefore contain
all Fiedler pencils modulo permutations. The important case of square matrix polynomials is examined
in greater depth. We prove that the intersection of any number of block Kronecker ansatz spaces is never
empty and construct large subspaces of block-symmetric matrix pencils among which still almost every pencil is a
strong
linearization. Moreover, we show that the original ansatz spaces $\mathbb{L}_1$ and $\mathbb{L}_2$ may essentially
be recovered
from block Kronecker ansatz spaces via pre- and postmultiplication, respectively, of certain constant matrices.
\end{abstract}


\end{frontmatter}

\section{Introduction}
The linearization of (non) square matrix polynomials
$$P(\lambda) = \sum_{i=0}^d P_i\lambda^i, P_i \in \mathbb{R}^{m \times n}$$
has received much attention in the last ten years,
motivated at least in part by the ground-breaking paper \cite{MacMMM06}. In that paper, three vector spaces
$\mathbb{L}_1$,
$\mathbb{L}_2$ and $\mathbb{DL}$ of potential linearizations (called \enquote{ansatz spaces}) for square matrix
polynomials
$P(\lambda) (m = n)$ have been introduced. The spaces  $\mathbb{L}_1$, $\mathbb{L}_2$ generalize the companion
form of the first and second kind, resp.,
$$
\mathbb{L}_1(P) = \{\mathcal{L}(\lambda) = \lambda X + Y \in \mathbb{R}[\lambda]^{nd \times nd} \mid
 \mathcal{L}(\lambda) \big( \Lambda_{d-1} \otimes I_n \big) = v \otimes P(\lambda) , v \in \mathbb{R}^d\},
$$
$\mathbb{L}_2(P) = \lbrace \mathcal{L}(\lambda)^T
\; | \; \mathcal{L}(\lambda) \in \mathbb{L}_1(P^T) \rbrace$
while the double ansatz space
\begin{equation}\label{DLP}
\mathbb{DL}(P) = \mathbb{L}_1(P) \cap \mathbb{L}_2(P)
\end{equation}
 is their intersection.
Here $\Lambda_j$ is the vector of the elements of the standard basis;
$\Lambda_j := \Lambda_j(\lambda) = [ \;
\lambda^{j} \; \lambda^{j - 1} \; \cdots \; \lambda \; 1 \; ]^T \in \mathbb{R}[\lambda]^{j + 1}$ for any integer $j
\geqslant 0.$ A  thorough discussion of these spaces can be found in \cite{MacMMM06} and \cite{HigMMT06}, see
\cite{DopLPVD16}
for
more references. In particular,
it is discussed in \cite{MacMMM06} that almost all pencils in these spaces are linearizations of $P(\lambda)$ and
in
\cite{HigMMT06} that any matrix pencil in $\mathbb{DL}(P)$ is block-symmetric.

The second main source of linearizations are Fiedler pencils $F_{\sigma}(\lambda)$. Unlike the linearizations from
the vector
spaces discussed above,
these can be defined not only for square, but also for rectangular matrices \cite{DeTDM12}. These pencils are defined in an
implicit way, either
in terms of products of matrices for square polynomials or as the output of a symbolic algorithm for rectangular
matrices,
 see \cite[Section 4]{DopLPVD16} for a definition, a summary of their properties and references to further work.

In \cite[Section 5]{DopLPVD16} the family of block Kronecker pencils is introduced, which include all of the
Fiedler pencils
(modulo permutations). For an arbitrary matrix pencil $M_0+\lambda M_1 \in \mathbb{R}^{(\eta+1)m \times
(\epsilon+1)n}$ any matrix
pencil of the form
\begin{equation}\label{blockKronpencil}
{\mathcal N}(\lambda) = \left[\begin{array}{c|c}
M_0+\lambda M_1 & L_\eta^T \otimes I_m\\ \hline
L_\epsilon \otimes I_n & 0_{\epsilon n \times\eta m}
\end{array}\right] \in \mathbb{R}^{((\eta + 1)m + \epsilon n) \times ((\epsilon + 1)n + \eta m)}
\end{equation}
is called an $(\epsilon,n,\eta,m)$-block Kronecker pencil, or simply, a block Kronecker pencil. Here,
 \begin{equation}
  L_{\kappa} = L_{\kappa}(\lambda) := \begin{bmatrix} -1 & \lambda & & & \\ & -1 & \lambda & & \\ & & \ddots &
\ddots & \\ & & &
-1 & \lambda
\end{bmatrix} \in \mathbb{R}[\lambda]^{\kappa \times (\kappa + 1)}. \label{Lkappa}
 \end{equation}
It is proven that ${\mathcal N}(\lambda)$  is a (strong) linearization of the matrix polynomial
$Q(\lambda) = (\Lambda_\eta(\lambda)^T\otimes I_m)(M_0+\lambda M_1)(\Lambda_\epsilon(\lambda)\otimes I_n)\in
\mathbb{R}[\lambda]^{m\times n}$ of degree $d \leq \epsilon+\eta+1.$

Inspired by the work in \cite{DopLPVD16}, we introduce a new family of equations for matrix pencils that may be
applied to square and rectangular matrix polynomials. Matrix pencils that satisfy one or more particular equations
form real
vector spaces that are shown to serve as an abundant source of strong linearization. Since these spaces share
important
properties with $\mathbb{L}_1$ and $\mathbb{L}_2$ and entirely contain all block Kronecker pencils as introduced in
\cite{DopLPVD16}, we named them ``block Kronecker ansatz spaces''. Our derivations based on these ansatz spaces are basically theoretically oriented. The purpose of this paper is twofold:
it builds a bridge between the two main linearization techniques - the ansatz space framework initiated in \cite{MacMMM06} and the approach via Fiedler pencils starting with \cite{AntV04} - along with the development of ansatz spaces in the style of \cite{MacMMM06} for rectangular matrix polynomials.

Although we define and introduce the block Kronecker ansatz spaces for rectangular matrix polynomials, we devote special attention to the investigation of the square case. In this context we are able to show that the intersection of any number of
block Kronecker ansatz spaces is never empty. As a main difference to $\mathbb{DL}$, pencils in two or more
block Kronecker
ansatz spaces are not block-symmetric in general but block-symmetric pencils form proper and large-dimensional
subspaces therein.
Still almost every matrix pencil, block-symmetric or not, is a strong linearization as long as the matrix polynomial under consideration is regular.
The main contribution of this paper is to provide a comprehensive introduction of block Kronecker
ansatz spaces, to prove their basic properties and to motivate these features by appropriately selected examples.
To this end, in order to focus on the essential ideas and concepts, we presents our results just for the real numbers $\mathbb{R}$. This enables us to concentrate on the precise introduction of the block Kronecker spaces (over $\mathbb{R}$) avoiding technicalities that might occur considering other fields.

After submission of the first version of this paper, the manuscript \cite{DoBPSZ16}
was released. In \cite{DoBPSZ16} the block Kronecker ansatz spaces have been introduced
independently as the family of extended block Kronecker pencils
motivated, as in our case, by the results in \cite{DopLPVD16}. However, the goal of
\cite{DoBPSZ16} is different than ours. While our goal is to establish a new ansatz space framework for the explicit construction of strong linearizations for matrix polynomials and to show the connections between those ansatz spaces, Fiedler pencils and block Kronecker pencils, the goal in \cite{DoBPSZ16} is to provide a unified approach to all the families of Fiedler-like pencils in any field via the more general concept of strong block minimal bases pencils. Being now aware of \cite{DoBPSZ16} we will reference to similar results throughout the paper and, moreover, point out some new insights taking the results from \cite{DoBPSZ16} into account.

The paper is organized as follows: in Section \ref{sec2} some basic notation and well-known results are reviewed. Section \ref{sec3} introduces the
block Kronecker ansatz space and its most important properties. Double ansatz spaces and their subspaces of
block-symmetric
pencils are considered in Section \ref{sec4},
while Section \ref{sec:L1L2} presents some further understanding of $\mathbb{L}_1$ and $\mathbb{L}_2$ based on our results. Some concluding remarks are given in Section
\ref{sec:conclusions}.

\section{Basic Notation}\label{sec2}
The following notation will be used throughout the paper: $I_n$ is the $n \times n$ identity matrix, $e_i$ its
$i$-th column and
$0_{m \times n}$ denotes
the $m \times n$ zero matrix. The Kronecker product of two matrices $A$ and $B$ is denoted $ A \otimes B$ whereas
the direct
product of $A$ and $B$ is $A \oplus B$, i.e. $A \oplus B = \textnormal{diag}(A,B)$. Whenever a $km \times kn$
matrix $A$
may be expressed as $A= \sum_{i,j=1}^k e_ie_j^T \otimes B_{ij}$ for certain $m \times n$ matrices $B_{ij}$, we call
$A^{\mathcal{B}} = \sum_{i,j=1}^k e_je_i^T \otimes B_{ij}$ the block-transpose of $A$ (see \cite[Def.
2.1]{HigMMT06}).
%\cite[Def. 3.1.2]{Mac06}).}e
For $\mathbb{R}[\lambda]$, the ring of real polynomials in the variable $\lambda$, the $m \times n$ matrix
ring over $\mathbb{R}[\lambda]$ is denoted by $\mathbb{R}[\lambda]^{m \times n}$. Its elements are referred to as
matrix
polynomials. Notice that $\mathbb{R}[\lambda]^{m \times n}$ is a vector space over $\mathbb{R}$.

Certainly, a matrix polynomial
$P(\lambda) \in \mathbb{R}[\lambda]^{m \times n}$ may always be expressed as
\begin{align} P(\lambda) &= P_d \lambda^d + P_{d-1} \lambda^{d-1} + \cdots + P_1 \lambda + P_0 \notag \\  &= [ \, P_d \; \,P_{d-1} \; \, \cdots \; \, P_0 \,](\Lambda_d(\lambda) \otimes I_n) \label{def_matrixpol}
\end{align}
for appropriate matrices $P_0, \ldots , P_d \in \mathbb{R}^{m \times n}$ and some $d \in \mathbb{N}$.

A matrix polynomial
$P(\lambda) \in \mathbb{R}[\lambda]^{m \times n}$ is called regular given the case $m = n$ and
$\textnormal{det}(P(\lambda))$ is
not identically zero. Otherwise, $P(\lambda)$ is called singular. A regular matrix polynomial $P(\lambda)$ is said
to be
unimodular if $\textnormal{det}(P(\lambda)) \in \mathbb{R}$. A scalar $z \in \mathbb{C}$ is referred to as a
(finite) eigenvalue
of $P(\lambda) \in \mathbb{R}[\lambda]^{m \times n}$, if $P(z)   \in \mathbb{C}^{m
\times n}$ is
singular. Its corresponding eigenspace is defined to be $\textnormal{null}(P(z))$, the nullspace of $P(z)$. \\

Two matrix polynomials $P(\lambda)$ and $Q(\lambda)$ are said to be unimodularly equivalent if there exist unimodular
matrices
$U(\lambda)$ and $V(\lambda)$ such that $P(\lambda) = U(\lambda) Q(\lambda)V(\lambda)$ holds. The equivalence is
called strict
whenever $U(\lambda)$ and $V(\lambda)$ may be chosen independent of $\lambda$. Given in the form
(\ref{def_matrixpol}), the matrix
polynomial $P(\lambda)$ has degree $k$, i.e. $\textnormal{deg}(P) = k$, whenever $P_k \neq 0$ and $P_i =
0$ for all $i >
k$. If $\textnormal{deg}(P)=1$ we refer to $P(\lambda)$ as a (matrix) pencil.
The subspace of all $m \times n$ matrix polynomials having at most degree $d \in \mathbb{N}$ is denoted
$\mathbb{R}_d[\lambda]^{m
\times n}$.
For any $P(\lambda) \in \mathbb{R}[\lambda]^{m \times n}$ and any $t \geq \textnormal{deg}(P)$, $t \in \mathbb{N}$, the $t$-reversal of
$P(\lambda)$ is defined
as the matrix polynomial
\begin{equation}
\textnormal{rev}_t(P(\lambda)) = \lambda^t P \left( \frac{1}{\lambda} \right) \in \mathbb{R}[\lambda]^{m \times n}.
\end{equation}
The matrix polynomial $P(\lambda)$ with $\textnormal{deg}(P) = k$ is said to have an infinite eigenvalue, if zero is an eigenvalue of
$\textnormal{rev}_k(P(\lambda))$. The corresponding eigenspace is $\textnormal{null}(\textnormal{rev}_k(P(0)))$.

\subsection{Linearizations of Matrix Polynomials}
A matrix pencil $\mathcal{L}(\lambda)$ is said to be a linearization of $P(\lambda) \in \mathbb{R}[\lambda]^{m
\times n}$ if
there exist two unimodular matrix polynomials $U(\lambda)$ and $V(\lambda)$ such that
\begin{equation} U(\lambda) \mathcal{L}(\lambda) V(\lambda) = \left[ \begin{array}{c|c} I_s & \\ \hline &
P(\lambda) \end{array}
\right] \label{def_linearization} \end{equation} holds for some $s \in \mathbb{N}_0$. Moreover, assuming
$\textnormal{deg}(P)=k$, the linearization
$\mathcal{L}(\lambda)$ is called strong whenever $\textnormal{rev}_1( \mathcal{L}(\lambda))$ is a linearization for
$\textnormal{rev}_k(P(\lambda))$ as well. It is a basic fact on strong linearizations that they preserve the finite
and infinite
elementary divisors of $P(\lambda)$ (see the information and the references given in \cite[Section 2]{DopLPVD16}
for more details). In particular, %this means, that
any strong linearization $\mathcal{L}(\lambda)$ of
$P(\lambda)$ has the same (finite and infinite) eigenvalues as $P(\lambda)$ and keeps on their algebraic and
geometric
multiplicities.

Given an $n
\times n$ matrix polynomial $P(\lambda) = \sum_{i=0}^k P_i \lambda^i$ of degree $\textnormal{deg}(P)=k$,
it is well
known, that the Frobenius companion form
$$ \textnormal{Frob}_P(\lambda) = \begin{bmatrix} P_k & & & \\ & I_n & & \\ & & \ddots & \\ & & & I_n \end{bmatrix}
\lambda +
\begin{bmatrix} P_{k-1} & \cdots & P_1 & P_0 \\ -I_n & & & \\ & \ddots & & \\ & & -I_n & \end{bmatrix} \in
\mathbb{R}_1[\lambda]^{kn \times kn}$$
is a strong linearization for $P(\lambda)$ no matter whether $P(\lambda)$ is regular or singular. Moreover,
%notice that
(strict) equivalence preserves (strong) linearizations.
%Notice that a
According to (\ref{def_linearization}) any matrix pencil is its own linearization. Thus, the notion of
linearization
does hardly make sense for matrix pencils. Since the construction of linearizations is our main concern throughout
the paper, we
will henceforth assume arbitrary matrix polynomials $P(\lambda)$ having degree $\textnormal{deg}(P) \geq
2$ to avoid the
potential occurrence of pathological cases.

\section{Block Kronecker Ansatz Spaces}\label{sec3}

The following definition introduces the main object of interest throughout the remaining paper. We will
consistently assume
$\epsilon$ and $\eta$ to be nonnegative integers.

  \begin{definition}[Block Kronecker Ansatz Equation] \label{def_GAS} \ \\
Let $P(\lambda)$ be an $m \times n$ matrix polynomial of degree $k= \epsilon + \eta + 1$.
We define $\mathbb{G}_{\eta + 1}(P)$ to be the set of all
$((\eta  + \hspace{0.01cm} 1)m  +  \epsilon n)  \times  ((\epsilon + 1)n + \eta m)$ matrix polynomials $\mathcal{L}(\lambda) = X \lambda + Y$ satisfying
%%%%% Generalized Ansatz Equation
%%%%%%%%%%%%%%%%%%%%%%%%%%%%%%%%%
\begin{equation} \big( (\Lambda_{\eta}(\lambda)^T \otimes I_m) \oplus I_{\epsilon n} \big) \mathcal{L}(\lambda)
\big ((\Lambda_{\epsilon}(\lambda) \otimes I_n) \oplus I_{\eta m} \big) = \alpha P(\lambda) \oplus 0_{\epsilon n \times \eta
m} \label{ansatzequation1} \end{equation}
for some $\alpha \in \mathbb{R}$. Equation (\ref{ansatzequation1}) is called block Kronecker ansatz equation for
the matrix polynomial $P(\lambda)$.
   \end{definition}
We will refer to $\mathbb{G}_{\eta + 1}(P)$ as a \enquote{block Kronecker ansatz space}
for $P(\lambda)$. This name was chosen in compliment of %the ground-breaking paper
the \enquote{ansatz spaces} established in \cite{MacMMM06} and the \enquote{block Kronecker pencils} introduced in
\cite{DopLPVD16}.
How the main ideas of both papers may
be unified via the concept of block Kronecker ansatz spaces is one primary concern of this paper.

\begin{remark}\label{rem1}
According to \cite[Def. 3.1, Thm. 3.3]{DopLPVD16} it is immediate that (\ref{ansatzequation1}) may be
formulated in the framework of dual minimal bases as well. Therefore, for any other pair of dual minimal bases
\cite[Def. 2.5]{DopLPVD16} a corresponding ansatz equation may be formulated and analyzed similar to our discussion
in the subsequent sections. However, most of the following results require that we know exactly how the dual
minimal bases look like. To this end, we confine ourselves to (\ref{ansatzequation1}).
\end{remark}

Notice that, since $\eta$ may take any integer value between $0$
and $k-1$, there always exist exactly $k$ block Kronecker ansatz spaces for $P(\lambda)$.
%%%%%% Lemma: Vector Space Property of GAS
\begin{lemma}[$\mathbb{G}_{\eta + 1}(P)$ is a $\mathbb{R}$-vector space] \label{lem_vectorspace}
For any $m \times n$ matrix polynomial $P(\lambda)$ of degree $k=\epsilon + \eta + 1$, $\mathbb{G}_{\eta + 1}(P)$
is a vector
space over $\mathbb{R}$.
\end{lemma}
Since the statement of Lemma \ref{lem_vectorspace} is quite obvious, we omit the proof.
Rather notice that equation (\ref{ansatzequation1}) may be reformulated as\footnote{In order to save space here and in subsequent formulas the dependence of $L_{\kappa}(\lambda)$ and $\Lambda_{\kappa}(\lambda)$ on $\lambda$ is sometimes omitted. Since there is no risk of confusion, $L_{\kappa}$ and $\Lambda_{\kappa}$ will always be understood as $L_{\kappa}(\lambda)$ and $\Lambda_{\kappa}(\lambda)$.}
\begin{equation} {\small \left[ \begin{array}{c|c} \Lambda_{\eta}^T \otimes I_m & 0 \\ \hline 0 & I_{\epsilon n}
\end{array}
\right]
 \left[ \begin{array}{c|c} \mathcal{L}_{11}(\lambda) & \mathcal{L}_{12}(\lambda) \\ \hline
\mathcal{L}_{21}(\lambda) &
\mathcal{L}_{22}(\lambda) \end{array} \right]
\left[ \begin{array}{c|c} \Lambda_{\epsilon} \otimes I_n & 0 \\ \hline 0 & I_{\eta m} \end{array} \right]
= \left[ \begin{array}{c|c} \alpha P(\lambda) & 0 \\ \hline 0 & 0_{\epsilon n \times \eta m} \end{array} \right] }
\label{ansatzequation2} \end{equation}
where we have expressed $\mathcal{L}(\lambda)$ as a $2 \times 2$ block matrix with the %corresponding
leading
$(\eta + 1)m \times (\epsilon + 1)n$ block $\mathcal{L}_{11}(\lambda)$. Following \cite[Def. 5.1]{DopLPVD16},
this structured $2 \times 2$
block-notation of $\mathcal{L}(\lambda) \in \mathbb{G}_{\eta + 1}(P)$ is called its natural partition. In terms of this
expression,
(\ref{ansatzequation2}) explicitly reads
\begin{equation}
 \left[ \begin{array}{c|c} (\Lambda_{\eta}^T \otimes I_m) \mathcal{L}_{11}(\lambda) (\Lambda_{\epsilon} \otimes
I_n) &
  ( \Lambda_{\eta}^T \otimes I_m) \mathcal{L}_{12}(\lambda) \\ \hline
  \mathcal{L}_{21}(\lambda) ( \Lambda_{\epsilon} \otimes I_n) & \mathcal{L}_{22}(\lambda) \end{array} \right] =
\left[
\begin{array}{c|c} \alpha P(\lambda) & 0 \\ \hline 0 & 0_{\epsilon n \times \eta m} \end{array} \right].
\label{ansatzequation3}
\end{equation}
For $(\Lambda_{\eta}^T \otimes I_m) \mathcal{L}_{11}(\lambda) (\Lambda_{\epsilon} \otimes I_n)$ we will steadily be
using the
short hand notation $\Phi( \mathcal{L}_{11}(\lambda))$ assuming the parameters involved in this
expression are clear from the context. For instance, (\ref{ansatzequation3}) implies
$\Phi(\mathcal{L}_{11}(\lambda)) =
\alpha P(\lambda)$.

Next we will consider the off-diagonal blocks of (\ref{ansatzequation3}).
Recall the definition of $L_{\kappa}(\lambda)$ (see (\ref{Lkappa})) and notice that $L_{\kappa}(\lambda) \Lambda_{\kappa}(\lambda) = 0$ (in fact
$L_{\kappa}(\lambda)$ and
$\Lambda_{\kappa}(\lambda)^T$ are dual
minimal bases, see \cite[Sec. 2]{DopLPVD16} for more information). Consequently
$(L_{\kappa}(\lambda) \otimes I_n)(\Lambda_{\kappa}(\lambda) \otimes I_n) = 0$ (see also \cite[Ex. 2.6]{DopLPVD16}).

\begin{lemma} \label{lem_nullspace}
Let $\mathcal{K}(\lambda)$ be an $\kappa_1 m \times (\kappa_2 + 1) n$ matrix pencil and assume
\begin{equation} \mathcal{K}(\lambda) \big(
\Lambda_{\kappa_2}(\lambda) \otimes I_n \big) = 0_{\kappa_1 m \times n}. \label{nullspace_equation} \end{equation}
Then
$\mathcal{K}(\lambda) = C(L_{\kappa_2}(\lambda) \otimes I_n)$ for some matrix  $C
\in \mathbb{R}^{\kappa_1 m \times \kappa_2 n}$.
\end{lemma}
\begin{proof}
Assume ${\mathcal{K}(\lambda)} =  [\, k_1 \; | \; K_1 \, ]\lambda + K_0$ with $k_1 \in \mathbb{R}^{\kappa_1 m
\times n}$
satisfies
(\ref{nullspace_equation}). Then
$$ \Delta \mathcal{K}(\lambda) = \mathcal{K}(\lambda) - K_1(L_{\kappa_2}(\lambda) \otimes I_n) =: [ \, d_1(\lambda)
\; | \; D_1
\, ] $$ is independent of $\lambda$ in all but its first block-column $d_1(\lambda) \in
\mathbb{R}_1[\lambda]^{\kappa_1m \times
n}$. However, from $\Delta \mathcal{K}(\lambda) ( \Lambda_{\kappa_2} (\lambda) \otimes I_n)$ we obtain
\begin{align*}  \mathcal{K}(\lambda) \big( \Lambda_{\kappa_2}(\lambda) \otimes I_n \big) - K_1 \big( L_{\kappa_2}(\lambda) \otimes I_n \big) \big( \Lambda_{\kappa_2}(\lambda) \otimes I_n \big) = 0_{\kappa_1 m \times n}, \end{align*}
so $\Delta \mathcal{K}(\lambda)$ still satisfies (\ref{nullspace_equation}).
Notice that $\Delta \mathcal{K}(\lambda) ( \Lambda_{\kappa_2} (\lambda) \otimes I_n)$ has dimension $\kappa_1 m \times n$ and that every $m \times n$ block is a matrix polynomial in the variables $1, \lambda , \lambda^2, \ldots , \lambda^{\kappa_2 +1}$. Due to the basis property
of the
monomials  this implies $\Delta \mathcal{K}(\lambda) \equiv 0$ and proves the statement.
\end{proof}

Further on, via block-transposition it can be seen    %accepted
that any $(\kappa_1 + 1)m \times
\kappa_2 n$ matrix pencil $\mathcal{K}(\lambda)$  satisfying $(\Lambda_{\kappa_1}(\lambda)^T \otimes
I_m)\mathcal{K}(\lambda) = 0$
has an expression $\mathcal{K}(\lambda) = (L_{\kappa_1}(\lambda)^T \otimes
I_m)C$ for some matrix $C \in \mathbb{R}^{\kappa_1 m \times \kappa_2 n}$. Hence, regarding (\ref{ansatzequation3})
once more,
we obtain
$$
 \mathcal{L}_{21}(\lambda) = C_1 \big( L_{\epsilon}(\lambda) \otimes I_n \big) \qquad \mathcal{L}_{12}(\lambda) =
\big( L_{\eta}(\lambda)^T \otimes I_m \big)C_2
$$
for matrices $C_1 \in \mathbb{R}^{\epsilon n \times \epsilon n}$ and $ C_2 \in \mathbb{R}^{\eta m \times \eta m}$.
Now, considering again the $(1,1)$-block in (\ref{ansatzequation3}) and an $m \times n$ matrix polynomial
$P(\lambda) =
\sum_{i=0}^k P_i \lambda^i$ of degree $k = \epsilon + \eta + 1$, observe that the $(\eta + 1)m  \times (\epsilon +
1)n$ matrix
pencil
$$
 \Sigma_{\eta,P} (\lambda) = \begin{bmatrix} \lambda P_k + P_{k-1} & P_{k-2} & \cdots &  P_{\eta } \\
               & & & P_{\eta - 1} \\ & 0_{\eta m \times \epsilon n} & & \vdots \\ & & & P_0
              \end{bmatrix} $$
satisfies $\Phi(\Sigma_{\eta , P}(\lambda)) = P(\lambda)$.
Therefore, for any other $(\eta + 1)m  \times (\epsilon + 1)n$ pencil $Q(\lambda)$ satisfying $\Phi( Q(\lambda)) =
\alpha
P(\lambda)$ for some $\alpha \in \mathbb{R}$
we obtain $$ \Phi \big( \alpha \Sigma_{\eta,P}(\lambda) - Q(\lambda) \big)
= \alpha \Phi \big( \Sigma_{\eta , P}(\lambda) \big) - \Phi \big( Q(\lambda) \big) = \alpha P(\lambda) - \alpha
P(\lambda) =
{0.}$$
Thus, interpreting $\Phi$ as a function mapping $(\eta + 1)m \times (\epsilon + 1)n$ matrix
pencils to $m \times n$ matrix polynomials $P(\lambda)$ of degree $\text{deg}(P) \leq \epsilon + \eta +
1$, $\Phi$ is
linear. Moreover, $\Phi$ is easily seen to be surjective. The homomorphism theorem gives
$$
 \mathbb{R}_1[\lambda]^{(\eta +1)m \times (\epsilon +1)n} / \textnormal{null}(\Phi) \, \cong \,
\mathbb{R}_{\epsilon +
\eta + 1}[\lambda]^{m \times n}
$$
and thus $\text{dim}(\text{null}(\Phi)) =\big( \eta( \epsilon + 1) + (\eta + 1) \epsilon \big)mn.$

Now note that the set $\mathcal{N}_{\epsilon , \eta}$ of all $(\eta + 1)m \times (\epsilon + 1)n$ matrix pencils $\mathcal{M}(\lambda)$ of the form
\begin{equation}
 \mathcal{M}(\lambda) = B_1 \big( L_{\epsilon}(\lambda) \otimes I_n \big) +  \big( L_{\eta}(\lambda)^T \otimes I_m
\big) B_2 \label{pencilMB1B2}
\end{equation}
with arbitrary matrices $B_1 \in \mathbb{R}^{(\eta + 1)m \times \epsilon n}$ and $B_2 \in \mathbb{R}^{\eta m \times
(\epsilon + 1)n}$ form a real vector space that is completely contained in
$\text{null}(\Phi)$. Following (\ref{pencilMB1B2}), the mapping $(B_1,B_2) \mapsto \mathcal{M}(\lambda)$ is injective since $\mathcal{M}(\lambda) = 0$ can only hold for $B_1=B_2=0$ (consider once more the form of $L_{\epsilon}(\lambda)$ and $L_{\eta}(\lambda)^T$, see (\ref{Lkappa})).

Therefore, we conclude that $\mathcal{N}_{\epsilon, \eta} = \textnormal{null}( \Phi)$ and
%of dimension $(\eta( \epsilon + 1) + (\eta + 1) \epsilon )mn$
obtain the following characterization of $\mathbb{G}_{\eta + 1}(P)$.

%%%%%% Theorem: Characterization Generalized Ansatz Space + Dimension
%%%%%%%%%%%%%%%%%%%%%%%%%%%%%%%%%%%%%%%%%%%%%%%%%%%%%%%%%%%%%%%%%%%%%
\begin{theorem}[Characterization of $\mathbb{G}_{\eta + 1}(P)$] \label{thm_generalspace} \ \\
Let $P(\lambda)$ an $m \times n$ matrix polynomial
of degree $k= \eta + \epsilon + 1$.
Then $\mathbb{G}_{\eta + 1}(P)$ is a vector space over $\mathbb{R}$ having dimension
%%%%%% Dimension Formula for GAS
$$
\textnormal{dim}( \mathbb{G}_{\eta + 1}(P)) = ( \epsilon n + \eta m)^2 + (\epsilon + \eta)mn + 1.
$$
Any matrix pencil $\mathcal{L}(\lambda) \in \mathbb{G}_{\eta + 1}(P)$ may be characterized as
%%%%%% Explicit Characterization GAS
   \begin{equation}
\mathcal{L}(\lambda) = {\small \left[
\begin{array}{c|c} \alpha \Sigma_{\eta ,P}(\lambda) + B_1(L_{\epsilon}(\lambda) \otimes I_n)
+ (L_{\eta}(\lambda)^T \otimes I_m)B_2 & (L_{\eta}(\lambda)^T \otimes I_m)C_2 \\ \hline C_1 (L_{\epsilon}(\lambda)
\otimes I_n) &
0 \end{array} \right] }\label{def_blockspace_explicit}
\end{equation}
with some $\alpha \in \mathbb{R}$ and some matrices $B_1 \in \mathbb{R}^{( \eta + 1)m \times \epsilon n}, B_2 \in
\mathbb{R}^{
\eta m \times ( \epsilon + 1)n}$, $C_1 \in \mathbb{R}^{ \epsilon n \times \epsilon n}$ and $ C_2 \in
\mathbb{R}^{\eta m \times
\eta m}$.
\end{theorem}

The dimension of $\mathbb{G}_{\eta + 1}(P)$ is just the sum of the dimensions of the constant matrices in
expression
(\ref{def_blockspace_explicit}) plus one for the scalar $\alpha$.
Moreover, note that any matrix pencil $\mathcal{L}(\lambda) \in \mathbb{G}_{\eta + 1}(P)$ of the form
(\ref{def_blockspace_explicit}) can be factorized uniquely as
 \begin{equation}
\mathcal{L}(\lambda) = \left[ \begin{array}{c|c}  I_{(\eta + 1)m} & B_1 \\ \hline 0 & C_1 \end{array} \right]
 \left[ \begin{array}{c|c} \alpha \Sigma_{\eta ,P}(\lambda) &  L_{\eta}(\lambda)^T \otimes I_m \\ \hline
L_{\epsilon}(\lambda) \otimes I_n & 0 \end{array} \right] \left[ \begin{array}{c|c} I_{(\epsilon +
1)n} & 0 \\ \hline B_2 & C_2 \end{array} \right].
\label{def_blockspace}
\end{equation}
Notice that this factorization is equivalent to (3.5) in \cite{DoBPSZ16}.
%\end{remark}

\begin{example} \label{ex_Gspace}
Let $P(\lambda) = \sum_{i=0}^6 P_i \lambda^i$ be an $m \times n$ matrix polynomial of degree
$\textnormal{deg}(P)=6$ and
consider the case $\eta = 3, \epsilon = 2$. According to (\ref{def_blockspace_explicit}) we
may construct the following matrix pencil
 $$ \small{ \mathcal{L}(\lambda) = \left[ \begin{array}{ccc|ccc} \lambda P_6 + P_5 & P_4 & P_3 & 0 & -F &
H \\ A &
-(B+ \lambda
A) & P_2 & 0 & E + \lambda F & -  \lambda H \\ -P_3 & \lambda B & P_1  & D & - \lambda E &
0
\\ \lambda P_3 & 0 & P_0 & - \lambda D & 0 & 0\\ \hline C & -(G + \lambda C) & \lambda G & 0 & 0 & 0
\\
0 & C & - \lambda C & 0 & 0 & 0 \end{array} \right] } $$
with arbitrary matrices $A, B \in \mathbb{R}^{m \times n}, C, G \in \mathbb{R}^{n \times n}$
and $D,E,F,H \in \mathbb{R}^{m \times m}$. It is not hard to see that $\mathcal{L}(\lambda) \in \mathbb{G}_{4}(P)$ since
$\mathcal{L}(\lambda)$ may be expressed in the form (\ref{def_blockspace}) with
$$\begin{bmatrix} B_1 \\ C_1 \end{bmatrix} = \left[ \begin{array}{cc} 0 & 0 \\ - A &
 0 \\ 0
&  0 \\ 0 & 0  \\ \hline  -C & G \\ 0 & -C \end{array}
 \right], \; \text{and} \; \begin{bmatrix} B_2 & C_2 \end{bmatrix} = \left[ \begin{array}{ccc|ccc} 0 & 0 & 0 & 0 &
F &
-H \\ 0 & B & 0 & 0 & E & 0 \\ P_3 & 0 & 0 & -D & 0 & 0 \end{array} \right], $$
and $\alpha=1$. As the next theorem
will reveal, $\mathcal{L}(\lambda)$ is a strong linearization for $P(\lambda)$ if $C$, $D, E$ and $H$ are all
nonsingular. In the
case of $P(\lambda)$ being square and regular, these three conditions turn out to be sufficient and necessary for
$\mathcal{L}(\lambda)$ being a strong linearization for $P(\lambda)$. Surprisingly, the choice of $A$ and $B$ does
not have any
effect in that regard.
\end{example}

The next theorem presents a quite
natural linearization condition for matrix pencils in block Kronecker ansatz spaces (see also \cite[Thm. 3.8]{DoBPSZ16}).
Notice that we a priori do not require $P(\lambda)$ to be regular or even square.

%%%%%% Theorem: Linearization Condition for GAS
%%%%%%%%%%%%%%%%%%%%%%%%%%%%%%%%%%%%%%%%%%%%%%%
\begin{theorem}[Linearization Condition for $\mathbb{G}_{\eta + 1}(P)$] \label{thm_lincondition} \ \\
Let $P(\lambda)$ be an $m \times n$ matrix polynomial and $\mathcal{L}(\lambda) \in \mathbb{G}_{\eta + 1}(P)$ as in
(\ref{def_blockspace}).
Then $\mathcal{L}(\lambda)$ is a strong linearization for $P(\lambda)$ if $\alpha \neq 0,$ and
\begin{equation}{\small
\left[ \begin{array}{c|c}  I_{(\eta + 1)m} & B_1 \\ \hline 0 & C_1 \end{array} \right] \in \textnormal{GL}_{(\eta +
1)m + \epsilon n}(\mathbb{R}), \; \, \text{and} \; \, \left[ \begin{array}{c|c} I_{(\epsilon +
1)n} & 0 \\ \hline B_2 & C_2 \end{array} \right] \in \textnormal{GL}_{(\epsilon + 1)n + \eta m}( \mathbb{R}).}
\label{equ_lincondition}
\end{equation}
Certainly (\ref{equ_lincondition}) is equivalent to $\textnormal{det}(C_1), \textnormal{det}(C_2) \neq 0$.
\end{theorem}

%%%%%% Proof
\begin{proof}
Assuming the matrices
   \begin{equation}
U = \left[ \begin{array}{c|c}  I_{(\eta + 1)m} & B_1 \\ \hline 0 & C_1 \end{array} \right] \; \, \text{and} \; \,
V = \left[ \begin{array}{c|c} I_{(\epsilon + 1)n} & 0 \\ \hline B_2 & C_2 \end{array} \right]
\label{matrices_UV}
   \end{equation}
are nonsingular, $\mathcal{L}(\lambda)$ in (\ref{def_blockspace}) is strictly equivalent to
\begin{equation}
\mathcal{F}_{\alpha , \eta , P}(\lambda) := \left[ \begin{array}{c|c} \alpha \Sigma_{\eta ,P}(\lambda) &
L_{\eta}(\lambda)^T \otimes I_m \\ \hline L_{\epsilon}(\lambda) \otimes I_n & 0 \end{array} \right].
\label{m_alpha}
\end{equation}
According to \cite[Thm. 5.2]{DopLPVD16} the matrix pencil $\mathcal{F}_{\alpha , \eta , P}(\lambda)$ is a strong
linearization for
$\alpha P(\lambda)$. Thus $\alpha \neq 0$ implies $\mathcal{F}_{\alpha , \eta , P}(\lambda)$ to be a strong
linearization for
$P(\lambda)$, so $\mathcal{L}(\lambda)$ is a strong linearization for $P(\lambda)$ as well.
\end{proof}

\begin{remark} \label{rem_linequiv}
Given the case of a regular $n \times n$ matrix polynomial $P(\lambda)$, the statement in Theorem
\ref{thm_lincondition}
becomes
an equivalence. In fact, if $\mathcal{L}(\lambda)$ as in (\ref{def_blockspace}) is a strong linearization for some
regular $P(\lambda)$, $\mathcal{L}(\lambda)$ is necessarily regular. This implies the matrices $U$ and $V$ to be
nonsingular and the
scalar $\alpha$ to be nonzero. However, for singular matrix polynomials $P(\lambda),$ (\ref{equ_lincondition}) is not necessary for $\mathcal{L}(\lambda)$ to be a strong linearization. For instance, consult \cite[Ex. 2]{DeTDM09} for an example of a strong linearization $\mathcal{L}(\lambda) \in \mathbb{G}_1(P)$ that does not satisfy (\ref{equ_lincondition}). A sufficient condition for strong linearizations in $\mathbb{G}_1(P)$ and $\mathbb{G}_k(P)$ of singular matrix polynomials $P(\lambda)$ is given in \cite[Sec. 5]{FassS16}.
\end{remark}

In \cite[Thm. 4.7]{MacMMM06} and \cite[Thm. 4.4]{DeTDM09}
it was shown that almost every pencil in $\mathbb{L}_1(P)$ (and $\mathbb{L}_2(P)$) is a strong linearization for
the (regular or singular) square matrix polynomial $P(\lambda)$. Here, a similar statement holds for
$\mathbb{G}_{\eta +
1}(P)$ and rectangular, i.e. not necessarily square
matrix polynomials $P(\lambda)$.

\begin{theorem}[Linearizations are Generic in $\mathbb{G}_{\eta + 1}(P)$] \label{thm_generic} \ \\
Let $P(\lambda)$ be an $m \times n$ matrix polynomial of degree $k = \epsilon + \eta + 1$.
Then almost every matrix pencil in $\mathbb{G}_{\eta + 1}(P)$ is a strong linearization for $P(\lambda)$.
\end{theorem}

Theorem \ref{thm_generic} follows directly from Theorem \ref{thm_lincondition} since $\mathbb{R} \setminus \lbrace
0 \rbrace$, $
\textnormal{GL}_{\epsilon n}(\mathbb{R})$
and $\textnormal{GL}_{\eta m}(\mathbb{R})$ are dense subsets of $\mathbb{R}$,
$\mathbb{R}^{\epsilon n \times \epsilon n}$
and $\mathbb{R}^{\eta m \times \eta m}$ respectively. Furthermore, notice that all the strong linearizations in
$\mathbb{G}_{\eta + 1}(P)$ are strong block minimal bases pencils, which have also been introduced in
\cite{DopLPVD16}.

Using \cite[Thm. 5.2]{DopLPVD16}, we now prove the \textit{Strong
Linearization Theorem} for block Kronecker ansatz spaces in the style of \cite[Thm. 4.3]{MacMMM06}. Showing the
connection
between the linearization property and the regularity of matrix pencils, we necessarily focus on regular (i.e.
square)
matrix polynomials.

%%%%%% Strong Linearization Theorem
\begin{theorem}[Strong Linearization Theorem for $\mathbb{G}_{\eta + 1}(P)$] \label{thm_master1} \ \\
Let $P(\lambda)$ be an $n \times n$ regular matrix polynomial and $\mathcal{L}(\lambda) \in \mathbb{G}_{\eta +
1}(P)$.
Then the following statements are equivalent
\begin{enumerate}
 \item $\mathcal{L}(\lambda)$ is a linearization for $P(\lambda)$.
 \item $\mathcal{L}(\lambda)$ is a regular matrix pencil.
 \item $\mathcal{L}(\lambda)$ is a strong linearization for $P(\lambda)$.
\end{enumerate}
\end{theorem}

\begin{proof}
 Since $3. \Rightarrow 1. \Rightarrow 2.$ is obvious, we only need to show $2. \Rightarrow 3.$ \\
 Assume $\mathcal{L}(\lambda)$ in (\ref{def_blockspace}) to be regular. This certainly
requires the nonsingularity of $U$ and $V$ as in (\ref{matrices_UV}) and
consequently implies the regularity of $\mathcal{F}_{\alpha , \eta , P}(\lambda)$.
Now suppose $\alpha = 0$. Then the ansatz equation (\ref{ansatzequation1}) gives
$$ \mathcal{F}_{0, \eta , P}(\lambda) \big( (\Lambda_{\epsilon} \otimes I_n) \oplus I_{\eta n} \big)e_i =  0 \quad
\text{and}
\quad
e_j^T \big( (\Lambda_{\eta}^T \otimes I_n) \oplus I_{\epsilon n} \big) \mathcal{F}_{0, \eta , P} = 0
$$
for any $1 \leqslant i,j \leqslant n$.
This shows that $\mathcal{F}_{0, \eta ,P}(\lambda)$ can not be regular, a contradiction. Therefore, the assumption
of
$\mathcal{L}(\lambda) \in
\mathbb{G}_{\eta + 1}(\lambda)$ being regular implies $\alpha \neq 0$ and thus the validity of all three conditions
in Theorem
\ref{thm_lincondition}.
\end{proof}

The next theorem shows that the eigenvector recovery for pencils in $\mathbb{G}_{\eta + 1}(P)$ is as easy as for
block Kronecker
pencils \cite[Section 7]{DopLPVD16}.

\begin{theorem} \label{thm_eigenvectors}
 Let $P(\lambda)$ be an $n \times n$ regular matrix polynomial of degree $k=\epsilon + \eta + 1$ and
$\mathcal{L}(\lambda) \in
\mathbb{G}_{\eta + 1}(P)$ be a strong linearization for $P(\lambda)$. Then the following statements hold (with $e_i
\in
\mathbb{R}^k$)
\begin{enumerate}
 \item If $u \in \mathbb{C}^{kn}$ is a right eigenvector of $\mathcal{L}(\lambda)$ with finite eigenvalue $\beta
\in
\mathbb{C}$, then $u^\star = (e_{\epsilon + 1}^T \otimes I_n)u$ is a right eigenvector of $P(\lambda)$
corresponding to the finite
eigenvalue
$\beta$.
\item If $u \in \mathbb{C}^{kn}$ is a right eigenvector of $\mathcal{L}(\lambda)$ with eigenvalue $\infty$, then
 $(e_1^T \otimes
I_n)u$ is a right eigenvector of $P(\lambda)$ with eigenvalue $\infty$.
 \item If $y \in \mathbb{C}^{kn}$ is a left eigenvector of $\mathcal{L}(\lambda)$ with finite eigenvalue $\beta \in
\mathbb{C}$, then $y^\star = (e_{\eta + 1}^T \otimes I_n)y$ is a left eigenvector of $P(\lambda)$
corresponding to the finite
eigenvalue
$\beta$.
\item If $y \in \mathbb{C}^{kn}$ is a left eigenvector of $\mathcal{L}(\lambda)$ with eigenvalue $\infty$, then
$(e_1^T \otimes
I_n)y$ is a left eigenvector of $P(\lambda)$ with eigenvalue $\infty$.
\end{enumerate}
\end{theorem}

\begin{proof}
Suppose $\mathcal{L}(\lambda) \in \mathbb{G}_{\eta + 1}(P)$ is given as in (\ref{def_blockspace}), i.e.
$\mathcal{L}(\lambda) = U \mathcal{F}_{\alpha, \eta , P}(\lambda) V$ using the notation of (\ref{matrices_UV}) and
(\ref{m_alpha}). Now assume $u \in \mathbb{C}^{kn}\backslash\{0\}$ satisfies $\mathcal{L}(\beta)u=0$ for some
$\beta \in
\mathbb{C}$. Then
$u^\star = \tfrac{1}{\alpha} Vu$ is a right eigenvector of $\mathcal{F}_{1, \eta , P}(\lambda)$ (recall that $U$ is
nonsingular,
i.e. $\textnormal{null}(U) = \emptyset$). Applying \cite[Thm.
7.6]{DopLPVD16} yields that $(e_{\epsilon + 1}^T \otimes I_n)u^\star$ is a right eigenvector of $P(\lambda)$ with
eigenvalue
$\beta$.
Now a closer look reveals $(e_{\epsilon + 1}^T \otimes I_n)u^\star = (e_{\epsilon + 1}^T
\otimes I_n)u$ due to the form of $V$. Thus $P(\beta)(e_{\epsilon + 1}^T \otimes I_n)u = 0$. The remaining
statements follow
by
exactly the same reasoning.
\end{proof}

Next, we provide a comprehensive example on block Kronecker pencils and
their connection to block Kronecker ansatz spaces.

\begin{example}[Block Kronecker Pencils] \ \\
Consider the set of matrix pencils $\mathcal{L}(\lambda)$ having the form (\ref{def_blockspace}) with $\alpha = 1$,
$C_1 = I_{\epsilon n},$ and $C_2 = I_{\eta m}$, i.e.,
\begin{align*}
 \mathcal{L}(\lambda) &= \left[ \begin{array}{c|c}  I_{(\eta + 1)m} & B_1 \\ \hline 0 & I_{\epsilon n} \end{array}
\right]
 \left[ \begin{array}{c|c} \alpha \Sigma_{\eta ,P}(\lambda) &  L_{\eta}(\lambda)^T \otimes I_m \\ \hline
L_{\epsilon}(\lambda) \otimes I_n & 0 \end{array} \right] \left[ \begin{array}{c|c} I_{(\epsilon +
1)n} & 0 \\ \hline B_2 & I_{\eta n} \end{array} \right] \\ &= \left[ \begin{array}{c|c} \Sigma_{\eta ,P}(\lambda) +
B_1(L_{\epsilon}(\lambda) \otimes I_n)
+ (L_{\eta}(\lambda)^T \otimes I_m)B_2 & (L_{\eta}(\lambda)^T \otimes I_m) \\ \hline (L_{\epsilon}(\lambda) \otimes
I_n) & 0
\end{array} \right].
\end{align*}
These matrix pencils coincide with the family of $(\epsilon , n, \eta ,
m)$-block Kro\-necker pencils (\ref{blockKronpencil}) that are strong linearizations for $P(\lambda)$. The strong linearization property was proven in \cite[Thm. 5.2]{DopLPVD16}, which complies with Theorem \ref{thm_lincondition} since in
this case
$\alpha
\neq 0$ and $C_1$ and $C_2$ are nonsingular.
\end{example}
\begin{remark}
For any arbitrary $m \times n$ matrix polynomial $P(\lambda)$, all $(\epsilon , n, \eta , m)$-block Kronecker
pencils are
elements of $\mathbb{G}_{\eta+1}(P)$. They do not form a vector subspace, but an affine subspace of
$\mathbb{G}_{\eta + 1}(P)$.
\end{remark}

It is stated in \cite[Sec. 4.2]{DopLPVD16} that for any Fiedler pencil
$F_{\sigma}(\lambda)$  there exist two permutation matrices $\Pi_1$ and
$\Pi_2$
such that $\Pi_1 F_{\sigma}(\lambda) \Pi_2$ is a block Kronecker pencil. Hence we may argue that
block Kronecker ansatz spaces contain all block Kronecker pencils and - modulo permutations - all Fiedler pencils.
Therefore,
based on \cite{DopLPVD16}, we succeeded in bringing together
Fiedler companion linearizations and ansatz spaces for the first time. In addition to that, it is shown in \cite{DoBPSZ16} that also the families of generalized Fiedler pencils, Fiedler pencils with repetition and generalized Fiedler pencils with repetition are - modulo permutations - elements of the block Kronecker ansatz spaces (introduced in \cite{DoBPSZ16} as the family of extended block Kronecker pencils). So, with rare exceptions, the block Kronecker ansatz spaces provide an extensive concept for the study of families of Fiedler-like pencils in combination with the ansatz space framework for the construction of linearizations known from \cite{MacMMM06}.

Moreover, we were able to make the idea of
ansatz spaces - which is, according to \cite{MacMMM06}, a concept valid for square matrix polynomials only -
 available for rectangular matrix polynomials as well.
However, notice that block Kronecker ansatz spaces contain infinitely many more matrix pencils then just permuted
Fiedler or
block Kronecker pencils. To this, it is a basic fact that every finite dimensional vector space as
$\mathbb{G}_{\eta + 1}(P)$ is
isomorphic to $\mathbb{R}^N$ for some $N \in \mathbb{N}_0$.
Inasmuch as $\mathbb{R}^N$ features a great many of analytical and topological properties, (\ref{def_blockspace})
strongly
suggests to define these concepts for $\mathbb{G}_{\eta + 1}(P)$ in terms of the pre- and postmultiplied
matrices and the
scalar $\alpha$. Taking this point of view, we may argue that the set of $(\epsilon , n,  \eta ,
m)$-block Kronecker pencils  constitutes a connected and nowhere dense subset in $\mathbb{G}_{\eta + 1}(P)$.

\section{Double Block Kronecker Ansatz Spaces $\mathbb{DG}_{\eta + 1}(P)$}\label{sec4}
In this section we characterize matrix pencils that belong to two or more block Kronecker ansatz spaces
simultaneously. Since
this
scheme does hardly seem promising in the case $m \neq n$, we confine ourselves to
square matrix polynomials.

This study is motivated by the double ansatz space $\mathbb{DL}(P)$ (\ref{DLP}). For any regular matrix polynomial
$P(\lambda)$ almost all pencils in $\mathbb{DL}(P)$ are linearizations of
$P(\lambda)$ \cite[Theorem 6.8]{MacMMM06}, while for singular $P(\lambda)$ none is a linearization \cite{DeTDM09}.
Moreover,
any matrix
pencil in $\mathbb{DL}(P)$ is block-symmetric which is in general not true for pencils in double block Kronecker
ansatz spaces.

\begin{definition}[Double Block Kronecker Ansatz Space] \ \\
Let $P(\lambda)$ be an $n \times n$ matrix polynomial of degree $k=\epsilon + \eta + 1$ and assume $\eta \leqslant
\epsilon$.
Then we define
$$
     \mathbb{DG}_{\eta + 1}(P) := \mathbb{G}_{\eta + 1}(P) \cap \mathbb{G}_{k - \eta}(P).
$$
\end{definition}

Given an $n \times n$ matrix polynomial $P(\lambda)$ of degree $k=\epsilon + \eta + 1$,
w.l.o.g. we will always assume $\eta \leqslant \epsilon = k - \eta - 1$ from now. This is reasonable since
     $$ \begin{aligned} \mathbb{DG}_{\eta + 1}(P) &= \mathbb{G}_{\eta + 1}(P) \cap \mathbb{G}_{k - \eta}(P)
      = \mathbb{G}_{k - \epsilon}(P) \cap \mathbb{G}_{\epsilon + 1}(P)
     = \mathbb{DG}_{\epsilon + 1}(P). \end{aligned} $$
Notice further that $\eta + 1 = k - \eta$ implies $k= 2
\eta + 1$. Therefore, the special case $\mathbb{DG}_{\eta + 1}(P) = \mathbb{G}_{\eta + 1} (P) \cap \mathbb{G}_{\eta
+ 1}(P)$ can
only occur for  $P(\lambda)$ having odd degree. Consider the following motivating example.

\begin{example} \label{ex_blockLspace1}
Let $P(\lambda) = \sum_{i=0}^6 P_i \lambda^i$ be an $n \times n$ matrix polynomial of degree
$\textnormal{deg}(P)=6$ and consider the case $\eta = 0$. Then
     \begin{equation} \mathcal{L}(\lambda) = \left[ \begin{array}{c|ccccc} \lambda P_6 + P_5 &
P_4 & P_3 & P_2 & P_1 & P_0 \\
 \hdashline
     P_4 & P_3 - \lambda P_4 & P_2 - \lambda P_3 & P_1 - \lambda P_2 & P_0 - \lambda P_1 & - \lambda P_0 \\
     P_3 & P_2 - \lambda P_3 & P_1 - \lambda P_2 & P_0 - \lambda P_1 & - \lambda P_0 & 0 \\
     P_2 & P_1 - \lambda P_2 & P_0 - \lambda P_1 & - \lambda P_0 & 0 & 0 \\
     P_1 & P_0 - \lambda P_1 & - \lambda P_0 & 0 & 0 & 0 \\
     P_0 & - \lambda P_0 & 0 & 0 & 0 & 0 \end{array} \right] \label{ex_DG} \end{equation}
is an element of $\mathbb{DG}_1(P) = \mathbb{G}_1(P) \cap \mathbb{G}_6(P).$\footnote{A closer look at the
block Kronecker ansatz
equation reveals, that $\mathbb{DG}_{1}(P)$
coincides with the subspace of all matrix pencils having a multiple of $e_1$ as ansatz vector in $\mathbb{DL}(P).$
We restrain
the study of the connection between the classical ansatz spaces $\mathbb{L}_1, \mathbb{L}_2$ and $\mathbb{DL}$ and
our approach
to Section \ref{sec:L1L2}.}
Further, $\mathcal{L}(\lambda)$ is a block-symmetric pencil. Now consider the case $\eta = 1$ and the
matrix pencil
     \begin{equation}
     \mathcal{K}(\lambda) = \left[ \begin{array}{cc|ccc:c} \lambda P_6 + P_5 & P_4 &
     A & 0 & -B & -I_n \\ 0 & P_3 & P_2 - \lambda A
     & P_1 & P_0 + \lambda B & \lambda I_n \\ \hdashline 0 & P_2 & P_1 - \lambda
     P_2 & P_0 - \lambda P_1 & - \lambda P_0 & 0 \\ C & P_1 - \lambda C & P_0 - \lambda P_1 & - \lambda P_0 & 0
     & 0 \\ 0 & P_0 & - \lambda P_0 & 0 & 0 & 0 \\ \hline -I_n & \lambda I_n & 0 & 0 & 0 & 0 \end{array}
     \right] \label{ex_DG2}
     \end{equation}
with arbitrary $n \times n$ matrices $A,B,C$. It is readily checked that $\mathcal{K}(\lambda) \in
\mathbb{DG}_{2}(P)$, i.e.
$\mathcal{K}(\lambda)$ is an element of
$\mathbb{G}_2(P)$ and $\mathbb{G}_5(P)$ simultaneously. Anyhow, it is obvious that $\mathcal{K}(\lambda)$ is
not block-symmetric.
\end{example}

Example \ref{ex_blockLspace1} shows that double block Kronecker ansatz spaces $\mathbb{DG}_{\eta + 1}(P)$ need not
contain
exclusively block-symmetric pencils. Albeit, they are never empty and the following theorem gives a comprehensive
characterization
of these spaces. To this end, we introduce a truncated square version of $\Sigma_{\eta,P}(\lambda)$, namely
$$ \Sigma_{\eta,P}^{\mathbb{DG}}(\lambda) =
\begin{bmatrix} \lambda P_k + P_{k-1} & P_{k-2} &\cdots & P_{\epsilon} \\
& & & \vdots \\
& & & P_{\epsilon - \eta} \end{bmatrix}
\in \mathbb{R}[\lambda]^{(\eta + 1)n \times (\eta + 1)n} $$
and set
$ \Pi_{\eta , P}^{\mathbb{DG}}(\lambda) := \big[ \, \Sigma_{\eta , P}^{\mathbb{DG}}(\lambda) \;
\mathcal{R}_{\eta,P} \,
\big] \in \mathbb{R}[\lambda]^{(\eta + 1)n \times (\epsilon + 1)n}$ with \begin{equation} \mathcal{R}_{\eta,P} =
\left[\begin{array}{c} 0_{\eta n \times (\epsilon -
\eta)n} \\
\hline \begin{array}{ccc} P_{\epsilon - \eta - 1} & \cdots & P_0 \end{array} \end{array} \right] \in
\mathbb{R}^{(\eta + 1)n
\times (\epsilon - \eta)n}. \label{Rmatrix} \end{equation}
Moreover, for $\epsilon \geqslant \eta$ we define the block Hankel matrix
$$
\mathcal{H}_{\epsilon - \eta}(P) = \begin{bmatrix} -P_{\epsilon - \eta -1} & \cdots & -P_1
& -P_0 \\ \vdots & \iddots & \iddots & \\[0.15cm] -P_1 & -P_0 & & \\[0.21cm]
-P_0 & & & \end{bmatrix} \in \mathbb{R}^{(\epsilon - \eta)n \times (\epsilon - \eta)n}.
$$
Notice that this block Hankel structure already showed up in the construction of block-symmetric linearizations in
\cite{HigMMT06}. We obtain
the following theorem.

\begin{theorem}[Characterization of $\mathbb{DG}_{\eta + 1}(P)$] \label{thm_char_DG} \ \\
Let $P(\lambda)$ be an $n \times n$ matrix polynomial of degree $k = \epsilon + \eta + 1$ and assume $\eta
\leqslant \epsilon$.
Then $\mathbb{DG}_{\eta + 1}(P)$ is a vector space over $\mathbb{R}$ having dimension
     \begin{equation}
     \textnormal{dim} \big( \mathbb{DG}_{\eta + 1}(P) \big) = 2k \eta n^2 + 1. \label{dimformula_DG}
     \end{equation}
Any matrix pencil $\mathcal{L}(\lambda) \in \mathbb{DG}_{\eta + 1}(P)$ may be characterized as
     \begin{equation}  \mathcal{L}(\lambda)= \left[ \begin{array}{c|c|c} I_{(\eta + 1)n} &  B_{11} & 0_{(\eta + 1)
     \times (\epsilon - \eta)n}  \\[0.1cm]  \hline 0 &  C_{11} & \alpha \mathcal{H}_{\epsilon - \eta}(P) \\ \hline 0 &
C_{21} &
     0_{\eta n \times (\epsilon - \eta )n} \end{array} \right]  \left[ \begin{array}{c|c} \alpha \Pi_{\eta ,
P}^{\mathbb{DG}}(\lambda)
     & L_{\eta}^T \otimes I_n \\ \hline L_{\epsilon} \otimes I_n & 0 \end{array} \right] \left[ \begin{array}{c|c}
I_{(\epsilon +
     1)n} & 0 \\[0.1cm] \hline B_{2} & C_{2}
     \end{array} \right] \label{doubleansatz_expr1}
     \end{equation}
with some $\alpha \in \mathbb{R}$ and some matrices $B_{11} \in \mathbb{R}^{(\eta + 1)n \times \eta n}$, $C_{11} \in
\mathbb{R}^{(\epsilon - \eta)n \times
\eta n}$, $C_{21} \in \mathbb{R}^{\eta n \times \eta n}$, $B_2 \in \mathbb{R}^{\eta n \times (\epsilon + 1)n}$ and
$C_2 \in
\mathbb{R}^{\eta n \times \eta n}$. Moreover, $\mathbb{DG}_{\eta + 1}(P)$ is a proper subspace of both
$\mathbb{G}_{\eta + 1}(P)$
and $\mathbb{G}_{k - \eta}(P)$.
\end{theorem}

\begin{proof}
Assume $P(\lambda)$ to be an $n \times n$ matrix polynomial of degree $k=\epsilon +
\eta + 1$ with $\eta \leqslant \epsilon$ and $\mathcal{L}^\star(\lambda)$ to be a $kn \times kn$
matrix pencil in $\mathbb{DG}_{\eta + 1}(P)$. Now consider $\mathcal{L}^\star(\lambda)$ partitioned as a $3\times 3 $
block
matrix as indicated in Figure \ref{fig1}
\begin{figure}[h]
\begin{center}
 \begin{tikzpicture}[scale=0.3]
  \draw[help lines, line width=0mm, color=black!10!white] (0,0) grid (12,12);
  \draw[line width=0.4mm] (0,0) -- (12,0) -- (12,12) -- (0,12) -- cycle;

  \draw[line width=0.4mm] (0,3) -- (12,3);
  \draw[line width=0.4mm] (9,0) -- (9,12);

  \draw[line width=0.4mm] (0,8) -- (12,8);

  \draw[line width=0.4mm] (4,3) -- (4,8);

  \draw[line width=0.4mm] (4,8) -- (4,12);
  \draw[line width=0.4mm] (4,0) -- (4,3);

\node[align=center] at (2,1.5)
{$(3,1)$};
\node[align=center] at (10.5,10)
{$(1,3)$};
\node[align=center] at (2,10) {$(1,1)$};

\node[align=center] at (6.5,1.5) {$(3,2)$};
\node[align=center] at (10.5,5.5) {$(2,3)$};
\node[align=center] at (10.5,1.5) {$(3,3)$};
\node[anchor=east] at (-0.3, 5.5) {$\mathcal{L}^\star(\lambda)=$};
\node[align=center] at (2,5.5) {$(2,1)$};
\node[align=center] at (6.5,10) {$(1,2)$};
\node[align=center] at (6.5,5.5) {$C(\mathcal{L}^\star)$};
%\scriptsize{$C_{12}(L_{\epsilon}^{(2)} \otimes I_n)$}
\fill[color=white, opacity=0.2] (0,0) rectangle (4,3);
\fill[color=white, opacity=0.2] (0,12) rectangle (4,8);
\fill[color=white, opacity=0.2] (9,8) rectangle (12,12);
\fill[color=white, opacity=0.2] (0,3) rectangle (3,8);
\fill[color=white, opacity=0.2] (4,8) rectangle (9,12);
\fill[color=white, opacity=0.2] (3,3) rectangle (4,8);
\fill[color=black!60!white, opacity=0.2] (4,3) rectangle (9,0);
\fill[color=black!60!white, opacity=0.2] (9,0) rectangle (12,8);

\draw[dotted] (0.1,-0.3) -- ++(0,-0.3) -- ++(3.8,0) -- ++(0,0.3);
\draw[dotted] (4.1,-0.3) -- ++(0,-0.3) -- ++(4.8,0) -- ++(0,0.3);
\draw[dotted] (9.1,-0.3) -- ++(0,-0.3) -- ++(2.8,0) -- ++(0,0.3);

\draw[dotted] (12.3,0.1) -- ++(0.3,0) -- ++(0,2.8) -- ++(-0.3,0);
\draw[dotted] (12.3,3.1) -- ++(0.3,0) -- ++(0,4.8) -- ++(-0.3,0);
\draw[dotted] (12.3,8.1) -- ++(0.3,0) -- ++(0,3.8) -- ++(-0.3,0);

\node[anchor=north] (a) at (2,-0.7) {\footnotesize{$(\eta + 1) n$}};
\node[anchor=north] (b) at (6.5,-0.7) {\footnotesize{$(\epsilon - \eta)n$}};
\node[anchor=north] (c) at (10.5,-1.1) {\footnotesize{$\eta n$}};

\node[anchor=west] (d) at (12.6, 1.5) {\footnotesize{$\eta n$}};
\node[anchor=west] (e) at (12.5, 5.5) {\footnotesize{$(\epsilon - \eta)n$}};
\node[anchor=west] (e) at (12.5, 10) {\footnotesize{$(\eta +1) n$}};

\end{tikzpicture}
\end{center}
\caption{$\mathcal{L}^\star(\lambda)$ in
its natural $3 \times 3$ partitioning. This partitioning may be interpreted as the overlay of the natural
partitionings of
elements in $\mathbb{DG}_{\eta + 1}(P)$ and $\mathbb{DG}_{k - \eta}(P)$. }
\label{fig1}
\end{figure}
as well as in its natural partitioning as a matrix pencil in $\mathbb{G}_{\eta+1}(P)$ in (\ref{ansatzequation2}).
The upper-left block $\mathcal{L}^\star_{11}(\lambda)$ is rectangular of size $(\eta+1)n\times (\epsilon+1)n,$
this corresponds to the $(1,1)$ and the $(1,2)$ blocks in the $3\times 3$ partitioning in Figure \ref{fig1}.
Clearly, the $(1,3)$ block corresponds to $\mathcal{L}^\star_{12}(\lambda),$ that is equal
to $(L_\eta(\lambda)^T\otimes I_n)C_2$ for a matrix $C_2 \in \mathbb{R}^{\eta n \times \eta n}$.
Moreover, from (\ref{def_blockspace_explicit}) it is obvious that $\mathcal{L}^\star_{22}(\lambda) \in
\mathbb{R}^{\epsilon n \times \eta n}$ is zero, thus the blocks $(2,3)$ and $(3,3)$ in Figure \ref{fig1} are zero.
Now consider $\mathcal{L}^\star(\lambda)$ in its natural partitioning as a matrix pencil in $\mathbb{G}_{k-\eta}(P).$
Then the block $\mathcal{L}^\star_{11}(\lambda)$ is rectangular of size $(\epsilon+1)n\times (\eta+1)n,$
this corresponds to the $(1,1)$ and the $(2,1)$ blocks in the $3\times 3$ partitioning in Figure \ref{fig1}.
Obviously, the $(3,1)$ block corresponds here to $\mathcal{L}^\star_{21}(\lambda),$ which is given as
$C_{21}(L_\eta(\lambda)\otimes I_n)$ for a matrix $C_{21} \in \mathbb{R}^{\eta n \times \eta n}.$
As before, the block $\mathcal{L}^\star_{22}(\lambda) \in  \mathbb{R}^{\eta n \times \epsilon n}$ is zero,
hence the blocks $(3,2)$ and $(3,3)$ in Figure \ref{fig1} are zero.
Thus, the fact of $\mathcal{L}^\star(\lambda)$ being an element of $\mathbb{G}_{\eta + 1}(P)$
and  of $\mathbb{G}_{k - \eta}(P)$ a priori implies the unalterable zero structure of $\mathcal{L}^\star(\lambda)$
in the blocks $(2,3), (3,2)$ and $(3,3)$ of the $3\times 3$ partitioning as indicated in grey in Figure \ref{fig1}.
In summary, we have identified all of the eight bordering blocks in Figure \ref{fig1}.
The remaining $(2,2)$-block in the $3\times 3$ partitioning, termed \enquote{core part} $C(
\mathcal{L}^\star)$ of $\mathcal{L}^\star(\lambda)$ in the following, is square of size
$(\epsilon-\eta)n\times (\epsilon-\eta)n.$ Our next step is to construct a pencil $\mathcal{L}(\lambda)$ of the
form
\[
\mathcal{L}(\lambda) = \left[ \begin{array}{c|c|c} I_{(\eta + 1)n} &
B_{11} & 0  \\[0.1cm]  \hline 0 & C_{11} & C_{12} \\ \hline 0 & C_{21} &
0  \end{array} \right]  \left[ \begin{array}{c|c} \alpha [\, \Sigma_{\eta, P}^{\mathbb{DG}}(\lambda) \; \, 0 \, ]
& L_{\eta}^T \otimes I_n
\\ \hline L_{\epsilon} \otimes I_n & 0 \end{array} \right]
\left[ \begin{array}{c|c} I_{(\epsilon + 1)n} & 0 \\ \hline
B_2 & C_2
\end{array} \right],
\]
such that the bordering blocks in $\Delta \mathcal{L}^\star(\lambda) := \mathcal{L}^\star(\lambda) -
\mathcal{L}(\lambda)$ get almost entirely eliminated. In fact, we may achieve that $\Delta
\mathcal{L}^\star(\lambda)$ has the
form indicated in Figure \ref{fig2} by making the appropriate choices of $C_{21}, C_2 \in
\mathbb{R}^{\eta n \times \eta n}$ as described above and finding suitable matrices $B_{11} \in \mathbb{R}^{(\eta +
1)n \times
\eta n}, C_{11} \in
\mathbb{R}^{(\epsilon - \eta)n \times \eta n},$ $C_{12} \in \mathbb{R}^{(\epsilon-\eta)n \times (\epsilon-\eta)n}$
and
$B_2 \in \mathbb{R}^{\eta n \times (\epsilon + 1)n}.$
That the core part of $\Delta \mathcal{L}^\star(\lambda)$  is equal to the core part $C(\mathcal{L}^\star)$ of
$\mathcal{L}^\star(\lambda)$ is achieved by setting the $C_{12}$-block of $\mathcal{L}(\lambda)$ as
$0_{(\epsilon-\eta)n}.$
Furthermore, there is a leftover matrix $h^\star \in
\mathbb{R}^{n \times (\epsilon - \eta)n}$ in the block (1,2) that can not be eliminated by $B_2$.

%%%%%%%%%%%%%%%%%%%%%%%%%%%%%%%% PIC 3 %%%%%%%%%%%%%%%%%%%%%%%%%%%%%%%%%%%%%%%%%
\begin{figure}[h]
\begin{center}
 \begin{tikzpicture}[scale=0.3]
  \draw[help lines, line width=0mm, color=black!10!white] (0,0) grid (12,12);
  \draw[line width=0.4mm] (0,0) -- (12,0) -- (12,12) -- (0,12) -- cycle;
  \draw[line width=0.4mm] (4,9) -- ++(5,0);
  \draw[line width=0.4mm] (0,3) -- (12,3);
  \draw[line width=0.4mm] (9,0) -- (9,12);

  \draw[line width=0.4mm] (0,8) -- (12,8);

  \draw[line width=0.4mm] (4,3) -- (4,8);

  \draw[line width=0.4mm] (4,8) -- (4,12);
  \draw[line width=0.4mm] (4,0) -- (4,3);

\node[align=center] at (2,1.5)
{\Large{$0$}};
\node[align=center] at (10.5,10)
{\Large{$0$}};
\node[align=center] at (2,10) {\Large{$0$}};

\node[align=center] at (6.5,1.5) {\Large{$0$}};
\node[align=center] at (10.5,5.5) {\Large{$0$}};
\node[align=center] at (10.5,1.5) {\Large{$0$}};
\node[anchor=east] at (-0.3, 5.5) {$\Delta \mathcal{L}^\star(\lambda)=$};
\node[align=center] at (2,5.5) {\Large{$0$}};
\node[align=center] at (6.5,10) {\Large{$0$}};
\node[align=center] at (6.5,5.5) {$C(\mathcal{L}^\star)$};
\fill[color=white, opacity=0.2] (0,0) rectangle (4,3);
\fill[color=white, opacity=0.2] (0,12) rectangle (4,8);
\fill[color=white, opacity=0.2] (9,8) rectangle (12,12);
\fill[color=white, opacity=0.2] (0,3) rectangle (3,8);
\fill[color=white, opacity=0.2] (4,9) rectangle (9,12);
\fill[color=black!60!white, opacity=0.2] (4,8) rectangle (9,9);
\node at (6.5,8.5) {\footnotesize{$h^\star$}};
\fill[color=white, opacity=0.2] (3,3) rectangle (4,8);
\fill[color=white, opacity=0.2] (4,3) rectangle (9,0);
\fill[color=white, opacity=0.2] (9,0) rectangle (12,8);

\draw[dotted] (0.1,-0.3) -- ++(0,-0.3) -- ++(3.8,0) -- ++(0,0.3);
\draw[dotted] (4.1,-0.3) -- ++(0,-0.3) -- ++(4.8,0) -- ++(0,0.3);
\draw[dotted] (9.1,-0.3) -- ++(0,-0.3) -- ++(2.8,0) -- ++(0,0.3);

\draw[dotted] (12.3,0.1) -- ++(0.3,0) -- ++(0,2.8) -- ++(-0.3,0);
\draw[dotted] (12.3,3.1) -- ++(0.3,0) -- ++(0,4.8) -- ++(-0.3,0);
\draw[dotted] (12.3,8.1) -- ++(0.3,0) -- ++(0,3.8) -- ++(-0.3,0);

\node[anchor=north] (a) at (2,-0.7) {\footnotesize{$(\eta + 1) n$}};
\node[anchor=north] (b) at (6.5,-0.7) {\footnotesize{$(\epsilon - \eta)n$}};
\node[anchor=north] (c) at (10.5,-0.7) {\footnotesize{$\eta n$}};

\node[anchor=west] (d) at (12.5, 1.5) {\footnotesize{$\eta n$}};
\node[anchor=west] (e) at (12.5, 5.5) {\footnotesize{$(\epsilon - \eta)n$}};
\node[anchor=west] (e) at (12.5, 10) {\footnotesize{$(\eta +1) n$}};
\fill[color=black!60!white, opacity=0.2] (4,0) -- ++(0,3) -- ++(5,0) -- ++(0,5) -- ++(3,0) -- ++(0,-8) -- cycle;

\end{tikzpicture}
\end{center}
\caption{$\Delta \mathcal{L}^\star(\lambda)$ in
its natural $3 \times 3$ partition.}
\label{fig2}
\end{figure}

Now consider the natural $2 \times 2$ block partition of $\Delta \mathcal{L}^\star(\lambda)$ as an element of
$\mathbb{G}_{\eta+1}(P)$ and in particular
$\Delta \mathcal{L}^\star_{11}(\lambda)$ (which corresponds to the $(1,1)$ and $(1,2)$ block in Figure \ref{fig2}).
Due to the
linearity of $\Phi$ we have
$$ \begin{aligned} \Phi\big( \Delta \mathcal{L}^\star_{11}(\lambda) \big)
&= \Phi( \mathcal{L}^\star_{11}(\lambda)) - \alpha  \Phi([\, \Sigma_{\eta, P}^{\mathbb{DG}}(\lambda) \;
0_{(\eta + 1)n \times (\epsilon
-\eta)n} \, ]) \\
&= \alpha P(\lambda) - \alpha \bigg( \sum_{i= \epsilon - \eta}^k P_i \lambda^i \bigg)
= \alpha \sum_{i=0}^{\epsilon - \eta - 1} P_{i} \lambda^{i}.
\end{aligned} $$
Considering once again Figure \ref{fig2}, this immediately implies $$h^\star = [ \, \alpha
P_{\epsilon - \eta - 1} \; \cdots \; \alpha P_0 \,].$$ Therefore, if we had chosen $\alpha \Pi_{\eta,
P}^{\mathbb{DG}}(\lambda)$
instead of $\alpha [\, \Sigma_{\eta, P}^{\mathbb{DG}}(\lambda) \; 0_{(\eta + 1)n \times (\epsilon + 1)n} \, ]$,
$h^\star$
would have also been deleted in $\Delta \mathcal{L}^\star(\lambda)$ as desired.

Now, since the $(\epsilon - \eta)n
\times (\epsilon - \eta)n$ core part $C(\mathcal{L}^\star)$ of $\mathcal{L}^\star(\lambda)$ has to be reproducible
in
both block Kronecker ansatz spaces, the choice $h^\star = [ \, \alpha P_{\epsilon - \eta - 1} \; \cdots \; \alpha
P_0 \,]$
unexpectedly
determines $C(\mathcal{L}^\star)$ completely. The unique possible form for $C( \mathcal{L}^\star)$ is
$$
C( \mathcal{L}^\star) = \alpha \begin{bmatrix}
P_{\mu - 1} -\lambda P_{\mu} & P_{\mu - 2} - \lambda P_{\mu - 1} &  \cdots & P_0 - \lambda P_1 & - \lambda
P_0 \\ P_{\mu - 2} - \lambda P_{\mu - 1} & P_{\mu - 3} - \lambda P_{\mu - 2}  & \iddots &
\iddots & \\ \vdots & \iddots  & \iddots  &  & \\ P_0 - \lambda P_1 & - \lambda P_0 & & & \\ - \lambda P_0 &
& & &
\end{bmatrix}
$$
where we have set $\mu := \epsilon - \eta - 1$ for abbreviation. Exactly this matrix pencil is obtained by setting
$C_{12} =
\alpha \mathcal{H}_{\epsilon -
\eta}(P)$. Therefore, we have shown that $\mathcal{L}^\star(\lambda)$ may be expressed as
     $$ \mathcal{L}^\star(\lambda)=\left[ \begin{array}{c|c|c} I_{(\eta + 1)n} & B_{11} & 0 \\  \hline 0 & C_{11} &
\alpha
\mathcal{H}_{\epsilon - \eta}(P) \\ \hline 0 & C_{21} & 0  \end{array} \right]  \left[
     \begin{array}{c|c} \alpha \Pi_{\eta , P}^{\mathbb{DG}}(\lambda) & L_{\eta}^T \otimes I_n \\ \hline
L_{\epsilon} \otimes I_n
     & 0 \end{array} \right] \left[ \begin{array}{c|c} I_{(\epsilon + 1)n} & 0 \\ \hline B_2 & C_2 \end{array}
\right] $$
which proves the statement. \end{proof}

\begin{corollary}[Non-Emptiness of $\mathbb{DG}_{\eta + 1}(P)$] \label{master2} \ \\
Let $P(\lambda)$ be an $n \times n$ matrix polynomial of degree $k= \eta + \epsilon + 1$ and assume $\eta \leqslant
\epsilon$.
Then
$$
     \mathbb{DG}_{\eta + 1}(P) \neq \emptyset.
$$
\end{corollary}

Recall the first case considered in Example \ref{ex_blockLspace1}. Note that Theorem \ref{thm_char_DG} shows that
$\mathcal{L}(\lambda)$ as in (\ref{ex_DG}) is - modulo scalar multiplication -
the only matrix pencil in $\mathbb{DG}_1(P)$ since we have $\textnormal{dim}(
\mathbb{DG}_1(P)) = 1$ according to (\ref{dimformula_DG}). Thus $\mathbb{DG}_{1}(P)$ consists entirely of
block-symmetric
pencils.\footnote{This is not surprising since $\mathbb{DG}_1(P)$ coincides with the subset of matrix pencils
having a multiple of $e_1$ as ansatz vector in $\mathbb{DL}(P)$. The vector space $\mathbb{DL}(P)$ contains
entirely
block-symmetric pencils. This was shown in \cite{HigMMT06}.} Regarding linearizations, the following fact can
immediately be
derived
from Theorem \ref{thm_lincondition} and Theorem
\ref{thm_master1} (see also Remark \ref{rem_linequiv}).

\begin{theorem}[Linearization Condition for $\mathbb{DG}_{\eta + 1}(P)$] \label{master4} \ \\
Let $P(\lambda)$ be a square and regular matrix polynomial of degree $k= \eta + \epsilon + 1$.
Let $\mathcal{L}(\lambda) \in \mathbb{DG}_{\eta + 1}(P)$  be given in the form (\ref{doubleansatz_expr1}).
Assume $\epsilon \neq \eta$. Then the
following statements are equivalent:
\begin{enumerate}
 \item $\mathcal{L}(\lambda)$ is a strong linearization for $P(\lambda)$.
 \item $P_0 \in \textnormal{GL}_n( \mathbb{R}), C_{21} \in \textnormal{GL}_{\eta n}(\mathbb{R}), C_2 \in
\textnormal{GL}_{\eta n}(\mathbb{R})$ and $\alpha \in \mathbb{R} \setminus \lbrace 0 \rbrace$.
\end{enumerate}
\end{theorem}
In the case $\epsilon = \eta$ the equivalence in Theorem \ref{master4} holds without the condition $P_0 \in \textnormal{GL}_n( \mathbb{R})$ in the second statement (due to the disappearance of the $\mathcal{H}$-block). In this case the implication $2. \Rightarrow 1.$ holds also for singular matrix polynomials. Moreover, note that the second equivalence in Theorem \ref{master4} is actually just a correspondingly adjusted version of Theorem \ref{thm_lincondition} that takes into account the special structure of pencils in $\mathbb{DG}_{\eta + 1}(P)$ (see (\ref{doubleansatz_expr1})). In particular, the condition $P_0 \in \textnormal{GL}_n( \mathbb{R})$ reflects the nonsingularity of $\mathcal{H}_{\epsilon - \eta}(P)$.

\begin{remark}
Theorem \ref{master4}, in the form given above,  can not be stated for singular matrix polynomials $P(\lambda).$ The second statement will never hold for singular $P(\lambda)$ since these always have a singular trailing coefficient $P_0$. This does a priori not mean that there can not be any linearizations for $P(\lambda)$ in this case, i.e. $1. \Rightarrow 2.$ certainly does not hold for singular matrix polynomials (see Remark \ref{rem_linequiv} and the reference therein).
\end{remark}

\begin{example}[Block Kronecker Pencils]
Notice that a pure block Kronecker pencil (\ref{blockKronpencil}) %as in \cite{DopLPVD16}
can never be an element of a double block Kronecker ansatz
space $$\mathbb{DG}_{\eta +
1}(P) =
\mathbb{G}_{\eta + 1}(P) \cap \mathbb{G}_{k - \eta}(P)$$ for any matrix polynomial $P(\lambda)$ unless $\eta + 1 =
k - \eta$.
Figuratively speaking, we need some connection between $\mathbb{G}_{\eta + 1}(P)$ and $\mathbb{G}_{k - \eta}(P)$ to
make a
pencil $\mathcal{L}(\lambda)$ an element of both spaces. The core part
$$  C( \mathcal{L}) =  \begin{bmatrix}  P_{\mu - 1} -
\lambda P_{\mu} & P_{\mu - 2} - \lambda P_{\mu - 1} &  \cdots & P_0 - \lambda P_1 & - \lambda
P_0 \\ P_{\mu - 2} - \lambda P_{\mu - 1} & P_{\mu - 3} - \lambda P_{\mu - 2}  & \iddots &
\iddots & \\ \vdots & \iddots  & \iddots  &  & \\ P_0 - \lambda P_1 & - \lambda P_0 & & & \\ - \lambda P_0 &
& & & \end{bmatrix} $$
with $\mu := \epsilon - \eta - 1$ takes on this task. Modulo a scalar multiplication, every pencil in
$\mathbb{DG}_{\eta + 1}(P)$ has the same core part, so it does essentially not depend on the specific pencil but on
the matrix
polynomial $P(\lambda)$. Moreover, $C( \mathcal{L})$ is block-symmetric. This block-symmetry turns out to be an
important property of pencils in double block Kronecker ansatz spaces and is further studied in the next section.
Notice that,
given the case
$\eta + 1 = k - \eta$, the core part vanishes entirely and no further restrictions remain for $\mathbb{DG}_{\eta +
1}(P)$. Only
in this situation we obtain pure block Kronecker pencils.
\end{example}

Consider once again Theorem \ref{master4}. The compliance of the irrevocable condition $P_0 \in \textnormal{GL}_n(
\mathbb{R})$
depends exclusively on the
matrix polynomial $P(\lambda)$ and holds if and only if zero is not an eigenvalue of $P(\lambda)$. On the other
hand, the
conditions $C_{21}, C_2 \in
\textnormal{GL}_{\eta n}(\mathbb{R})$ are satisfied for almost every matrix in $\mathbb{R}^{\eta n
\times \eta n}$. Since the implication $2. \Rightarrow 1.$ in Theorem \ref{master4} holds without the assumption
of
regularity (according to Theorem \ref{thm_lincondition}), we obtain the following general density property.

\begin{corollary}[Linearizations are Generic in $\mathbb{DG}_{\eta + 1}(P)$] \label{generic1} \ \\
Let $P(\lambda)$ be a square
matrix polynomial and assume zero is not an eigenvalue of $P(\lambda)$.
Then almost every matrix pencil in $\mathbb{DG}_{\eta + 1}(P)$ is a strong linearization for $P(\lambda)$.
\end{corollary}

 \begin{remark}
Assume $\eta = 0$ and consider $\mathbb{DG}_1(P)$. Then Theorem \ref{master4} reduces to the Eigenvalue Exclusion
Theorem (see
\cite[Thm. 6.7]{MacMMM06}) which is a powerful tool in the study of the space $\mathbb{DL}(P)$. %Since
%$\mathbb{DG}_1(P) \subset
%\mathbb{DL}(P)$
It states in this particular case, that $\mathcal{L}(\lambda) \in \mathbb{DG}_1(P)$ is a strong
linearization for $P(\lambda)$ if and only if no root of the $v$-polynomial  $$p( \lambda ;
\alpha e_1) = \alpha \lambda^{k-1}$$
(see \cite[Def. 6.1]{MacMMM06}) is an eigenvalue of $P(\lambda)$. Since $0$ is the only root of $p( \lambda ;
\alpha e_1)$
this means that $P(0) = P_0$ has to be
nonsingular, i.e. $P_0 \in \textnormal{GL}_n( \mathbb{R})$. Moreover, because the matrices $C_{21}$ and $C_2$ vanish
completely
(see (\ref{ex_DG}) in Example \ref{ex_blockLspace1})
this is the only condition to hold for $\mathcal{L}(\lambda) \in \mathbb{DG}_1(P)$ being a strong
linearization for $P(\lambda)$.
 \end{remark}

\subsection{The Superpartition Principle}
Although double block Kronecker ansatz spaces usually do not contain solely block-symmetric pencils, they possess a
remarkable
feature that we call \enquote{superpartition property}. This property was also recognized by the authors of \cite{DoBPSZ16} and mentioned in their Remark 3.3. To its motivation, consider the following example.

\begin{example} \label{ex_blockspace}
Let $P(\lambda) = \sum_{i=0}^6 P_i \lambda^i$ be an $n \times n$ matrix polynomial of degree
$\textnormal{deg}(P)=6$.
Consider as in Example \ref{ex_blockLspace1} the case $\eta = 1$ ($\epsilon = 4$) and the corresponding matrix pencil
 $\mathcal{K}(\lambda)$
$$
\left[ \begin{array}{cc|ccc:c} \lambda P_6 + P_5 & P_4 &
     A & 0 & -B & -I_n \\ 0 & P_3 & P_2 - \lambda A
     & P_1 & P_0 + \lambda B & \lambda I_n \\ \hdashline 0 & P_2 & P_1 - \lambda
     P_2 & P_0 - \lambda P_1 & - \lambda P_0 & 0 \\ C & P_1 - \lambda C & P_0 - \lambda P_1 & - \lambda P_0 & 0
     & 0 \\ 0 & P_0 & - \lambda P_0 & 0 & 0 & 0 \\ \hline -I_n & \lambda I_n & 0 & 0 & 0 & 0 \end{array}
     \right].
$$
As already discussed, $\mathcal{K}(\lambda) \in \mathbb{DG}_{2}(P) = \mathbb{G}_2(P) \cap \mathbb{G}_5(P).$
Now consider $\mathcal{K}(\lambda)$ in the
slightly modified partitioned form
     \begin{equation}
%     \mathcal{K}(\lambda) =
\left[ \begin{array}{ccc|c:cc} \lambda P_6 + P_5 & P_4 &
     A & 0 & -B & -I_n \\ 0 & P_3 & P_2 - \lambda A
     & P_1 & P_0 + \lambda B & \lambda I_n \\ 0 & P_2 & P_1 - \lambda
     P_2 & P_0 - \lambda P_1 & - \lambda P_0 & 0 \\ \hdashline  C & P_1 - \lambda C & P_0 -
     \lambda P_1 & - \lambda P_0 & 0 & 0 \\ \hline 0 & P_0 & - \lambda P_0 & 0 & 0 & 0 \\  -I_n & \lambda I_n & 0 &
0 & 0
     & 0 \end{array} \right]. \label{example_DG2}
     \end{equation}
It is readily checked, that $\mathcal{K}(\lambda)$ partitioned as in (\ref{example_DG2})  may alternatively be taken
as an
element of $\mathbb{G}_3(P)$ and of $\mathbb{G}_4(P)$ (e.g. $\eta = 2$, $\epsilon =3$). In other words,
$\mathcal{K}(\lambda) \in
\mathbb{DG}_3(P)$.
\end{example}

The next theorem states that the phenomenon highlighted in Example \ref{ex_blockspace} always holds (see also \cite[Thm. 3.10]{DoBPSZ16}). The main
reason behind this
fact is easily seen to be the block-symmetric core part of pencils in double block Kronecker ansatz spaces.

\begin{theorem}[Superpartition Property of $\mathbb{DG}_{\eta + 1}(P)$] \label{thm_interpolation} \ \\
Let $P(\lambda)$ be an $n \times n$ matrix polynomial of degree $k=\epsilon + \eta + 1$ and assume \\$\eta \leqslant
\epsilon$.
Then $\mathcal{L}(\lambda) \in
\mathbb{DG}_{\eta + 1}(P)$ implies that $\mathcal{L}(\lambda) \in \mathbb{G}_{\eta + i}(P)$ for
all $i = 1, 2, \ldots , k - 2 \eta.$
\end{theorem}

For ease of notation in the proof of this theorem we introduce the following partitioning of $\mathcal{L}(\lambda)$
\[\mathcal{L}(\lambda) =
 \left[ \begin{array}{c|c} \mathcal{L}_{11}^{(i)} (\lambda) & \mathcal{L}_{12}^{(i)}(\lambda) \\ \hline
\mathcal{L}_{21}^{(i)}(\lambda) &
\mathcal{L}_{22}^{(i)}(\lambda) \end{array} \right], \quad \mathcal{L}_{11}^{(i)} (\lambda)  \in
\mathbb{R}^{(\widetilde{\eta}+1)n \times (\widetilde{\epsilon}+1)n}, ~~\widetilde{\eta} = \eta +i,
\widetilde{\epsilon} =
\epsilon
-i.
\]
The condition $k = \eta + \epsilon +1 = \widetilde{\eta} + \widetilde{\epsilon}+1$ has to hold.
For $i=0$ this is the natural partition  (\ref{ansatzequation2})  considered so far; in particular,
$\mathcal{L}_{11}^{(0)} (\lambda) = \mathcal{L}_{11} (\lambda).$
Increasing $i$ by one, the upper-left $(1,1)$-block of $\mathcal{L}(\lambda)$ is increased by one block row and
decreased by one
block column.

\begin{remark}
 Due to the assumption $\mathcal{L}(\lambda) \in \mathbb{DG}_{\eta + 1}(P)$, it suffices to show that
$\Phi(\mathcal{L}_{11}^{(i)}(\lambda)) = \alpha P(\lambda)$ for all $i=1, \ldots , k - 2 \eta$ holds in order to prove
Theorem \ref{thm_interpolation}. To see this, consider exemplarily a matrix polynomial $P(\lambda) = \sum_{i=0}^7 P_i \lambda^i$ of degree $\textnormal{deg}(P)=7$ with $\eta = 1$. According to Theorem \ref{thm_char_DG} any pencil $\mathcal{L}(\lambda) \in \mathbb{DG}_{2}(P)$ schematically has the form
\begin{center}
  \begin{tikzpicture}[scale=0.6]
\draw[line width=0.05mm] (0,0) grid (7,7);
\node at (6.5,0.5) {0};
\node at (5.5,0.5) {0};
\node at (4.5,0.5) {0};
\node at (3.5,0.5) {0};
\node at (2.5,0.5) {0};
\node at (6.5,1.5) {0};
\node at (5.5,1.5) {0};
\node at (4.5,1.5) {0};
\node at (3.5,1.5) {0};
\node at (6.5,2.5) {0};
\node at (5.5,2.5) {0};
\node at (4.5,2.5) {0};
\node at (6.5,3.5) {0};
\node at (5.5,3.5) {0};
\node at (6.5,4.5) {0};
%\node at (2.5,5.5) {\scriptsize $P_3$};
%\node at (1.5,4.5) {\scriptsize $P_3$};
%\node at (1.5,3.5) {\scriptsize $P_2$};
%\node at (1.5,2.5) {\scriptsize $P_1$};
%\node at (1.5,1.5) {\scriptsize $P_0$};
%\node at (3.5,5.5) {\scriptsize $P_2$};
%\node at (4.5,5.5) {\scriptsize $P_1$};
%\node at (5.5,5.5) {\scriptsize $P_0$};
\node at (3.5,3.5) {core part};
\draw[thick] (0,5) -- (7,5);
\draw[thick] (6,0) -- (6,7);
\draw[thick] (2,0) -- (2,7);
\draw[thick] (0,1) -- (7,1);
\fill[black!30, opacity=0.3] (0,7) rectangle (6,5);
\fill[black!30, opacity=0.3] (6,0) rectangle (7,5);
\fill[black!30, opacity=0.3] (0,5) rectangle (2,1);
\fill[black!30, opacity=0.3] (2,1) rectangle (6,0);
\end{tikzpicture}
\end{center}
with the indicated unalterable zero-structure and the $3 \times 3$ partitioning as in Figure 1 and 2. The following sketches indicate the natural partitioning (\ref{ansatzequation2})
of pencils in the block Kronecker ansatz spaces $\mathbb{DG}_{\kappa}(P)$, $\kappa = 3,4,5,6$ applied to the pencil $\mathcal{L}(\lambda)$:
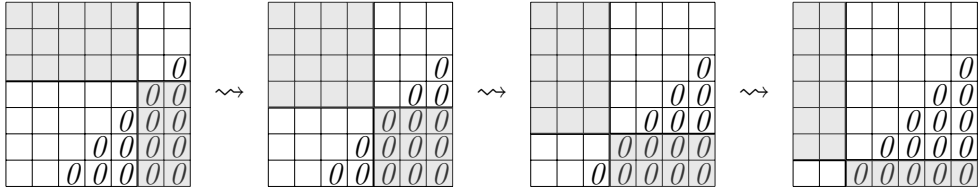
\begin{center}
\begin{tikzpicture}[scale=0.35]
\draw[line width=0.05mm] (0,0) grid (7,7);
\node at (6.5,0.5) {0};
\node at (5.5,0.5) {0};
\node at (4.5,0.5) {0};
\node at (3.5,0.5) {0};
\node at (2.5,0.5) {0};
\node at (6.5,1.5) {0};
\node at (5.5,1.5) {0};
\node at (4.5,1.5) {0};
\node at (3.5,1.5) {0};
\node at (6.5,2.5) {0};
\node at (5.5,2.5) {0};
\node at (4.5,2.5) {0};
\node at (6.5,3.5) {0};
\node at (5.5,3.5) {0};
\node at (6.5,4.5) {0};
\draw[thick] (0,4) -- (7,4);
\draw[thick] (5,0) -- (5,7);
\fill[black!30, opacity=0.3] (0,7) rectangle (5,4);
\fill[black!30, opacity=0.3] (5,4) rectangle (7,0);
\node at (8.5,3.5) {$\rightsquigarrow$};
\end{tikzpicture}
\begin{tikzpicture}[scale=0.35]
\draw[line width=0.05mm] (0,0) grid (7,7);
\node at (6.5,0.5) {0};
\node at (5.5,0.5) {0};
\node at (4.5,0.5) {0};
\node at (3.5,0.5) {0};
\node at (2.5,0.5) {0};
\node at (6.5,1.5) {0};
\node at (5.5,1.5) {0};
\node at (4.5,1.5) {0};
\node at (3.5,1.5) {0};
\node at (6.5,2.5) {0};
\node at (5.5,2.5) {0};
\node at (4.5,2.5) {0};
\node at (6.5,3.5) {0};
\node at (5.5,3.5) {0};
\node at (6.5,4.5) {0};
\draw[thick] (0,3) -- (7,3);
\draw[thick] (4,0) -- (4,7);
\fill[black!30, opacity=0.3] (0,7) rectangle (4,3);
\fill[black!30, opacity=0.3] (4,3) rectangle (7,0);
\node at (8.5,3.5) {$\rightsquigarrow$};
\end{tikzpicture}
\begin{tikzpicture}[scale=0.35]
\draw[line width=0.05mm] (0,0) grid (7,7);
\node at (6.5,0.5) {0};
\node at (5.5,0.5) {0};
\node at (4.5,0.5) {0};
\node at (3.5,0.5) {0};
\node at (2.5,0.5) {0};
\node at (6.5,1.5) {0};
\node at (5.5,1.5) {0};
\node at (4.5,1.5) {0};
\node at (3.5,1.5) {0};
\node at (6.5,2.5) {0};
\node at (5.5,2.5) {0};
\node at (4.5,2.5) {0};
\node at (6.5,3.5) {0};
\node at (5.5,3.5) {0};
\node at (6.5,4.5) {0};
\draw[thick] (0,2) -- (7,2);
\draw[thick] (3,0) -- (3,7);
\fill[black!30, opacity=0.3] (0,7) rectangle (3,2);
\fill[black!30, opacity=0.3] (3,2) rectangle (7,0);
\node at (8.5,3.5) {$\rightsquigarrow$};
\end{tikzpicture}
\begin{tikzpicture}[scale=0.35]
\draw[line width=0.05mm] (0,0) grid (7,7);
\node at (6.5,0.5) {0};
\node at (5.5,0.5) {0};
\node at (4.5,0.5) {0};
\node at (3.5,0.5) {0};
\node at (2.5,0.5) {0};
\node at (6.5,1.5) {0};
\node at (5.5,1.5) {0};
\node at (4.5,1.5) {0};
\node at (3.5,1.5) {0};
\node at (6.5,2.5) {0};
\node at (5.5,2.5) {0};
\node at (4.5,2.5) {0};
\node at (6.5,3.5) {0};
\node at (5.5,3.5) {0};
\node at (6.5,4.5) {0};
\draw[thick] (0,1) -- (7,1);
\draw[thick] (2,0) -- (2,7);
\fill[black!30, opacity=0.3] (0,7) rectangle (2,1);
\fill[black!30, opacity=0.3] (2,1) rectangle (7,0);
\end{tikzpicture}
\end{center}
This example shows, that the zero-structure of a pencil in $\mathbb{DG}_2(P)$ is exactly of the form that it covers all the $(2,2)$-zero blocks of pencils in $\mathbb{DG}_{\kappa}(P)$ with $2 \leq \kappa \leq 6$. Moreover, due to the special construction of the core part $C( \mathcal{L})$, the (1,2) and (2,1)-corner blocks as well as the upper-left (1,1)-block in the sketches above are always reproducible in every ansatz space $\mathbb{DG}_{\kappa}(P)$ for $2 \leq \kappa \leq 6$. Since the situation is exactly the same for other degrees of $P(\lambda)$ we only need to show that $\Phi(
\mathcal{L}_{11}^{(i)}(\lambda)) = \alpha P(\lambda)$ holds for all $i=1, \ldots , k - 2 \eta$ to prove Theorem \ref{thm_interpolation}.
That the latter holds is once more a consequence of the form of $C( \mathcal{L})$.
\end{remark}

\begin{proof}[Proof (Theorem \ref{thm_interpolation})]
First of all, according to (\ref{doubleansatz_expr1}),
$\mathcal{L}_{11}(\lambda)$ may be expressed as
$$ \mathcal{L}_{11}(\lambda) = \big[ \, B_{11} \; 0 \; \big]( L_{\epsilon} \otimes I_n) + (L_{\eta}^T \otimes
I_n) \big[ \,
B_{21} \; B_{22} \, \big] + \big[ \; \Sigma_{\eta, P}^{\mathbb{DG}} \; \mathcal{R}_{\eta,P} \; \big]$$
with $B_2 = [ \, B_{21} \; B_{22} \, ], B_{21} \in \mathbb{R}^{\eta n \times (\eta + 1)n}$ and $\mathcal{R}_{\eta , P}$ as in (\ref{Rmatrix}). Then,
we obtain that  $\mathcal{L}_{11}^{(i)}(\lambda) $
may be expressed  as
\footnotesize{ $$
{\mathcal{L}_{11}^{(i)}(\lambda) = \left[ \begin{array}{c|c} \begin{array}{c} B_{11} \\ \hline C_{1,i} \end{array}
& 0
\end{array} \right] (L_{\widetilde{\epsilon}} \otimes I_n)  + (L_{\widetilde{\eta}}^T \otimes I_n) \left[
\begin{array}{c|c}
B_{21} & B_{22,i} \\ \hline 0 & \alpha H_{\epsilon - \eta}^{(i)}(P)  \end{array} \right] +  \alpha \Omega_{\eta +
i,P}(\lambda)}
$$ }
\normalsize
with the $(\widetilde{\eta}+1)n \times (\widetilde{\epsilon}+1)n$ matrix pencil $\Omega_{\eta + i,P}(\lambda)$
\begin{equation}
\Omega_{\eta + i,P}(\lambda) = \left[ \begin{array}{ccc|ccc}  \lambda P_k + P_{k-1} & \cdots & P_{\epsilon} & & &
\\ & &
\vdots & & & \\ & & P_{\epsilon - \eta} & & & \\ \hline & & P_{\epsilon - \eta - 1} & & & \\ & & \vdots & & & \\ &
&
P_{\epsilon - \eta - i} & P_{\epsilon - \eta - i - 1} & \cdots & P_0
\end{array} \right].
\label{matrix_Omega}
\end{equation}
Here, $\mathcal{H}_{\epsilon - \eta}^{(i)}(P)$ denotes the upper left $in \times (\epsilon - \eta - i)n$ submatrix
of
$\mathcal{H}_{\epsilon - \eta}(P)$, $C_{1,i}$ the first $in$ rows of $C_{1}$, i.e. $C_{1,i} \in \mathbb{R}^{in
\times
(\epsilon - \eta)n}$, and $B_{22,i}$ the matrix $B_2$ missing the last $in$ columns, i.e. $B_{22,i} \in
\mathbb{R}^{\eta n \times
(\epsilon - \eta - i)n}$. Now, since $\Phi(\Omega_{\eta + i,P}(\lambda)) = P(\lambda)$ holds we obtain $\Phi(
\mathcal{L}_{11}^{(i)}(\lambda)) = \alpha P(\lambda)$.
\end{proof}
\begin{remark}
According to Example \ref{ex_blockspace} it is not surprising, that Theorem \ref{thm_interpolation} holds. The
property of a
matrix pencil $\mathcal{L}(\lambda)$ being an
element of $\mathbb{DG}_{\eta + 1}(P)$ imposes several restrictions on the form of $\mathcal{L}(\lambda)$. In
particular,
whereas the bordering blocks in the $3\times 3$ partitioning as in Figure \ref{fig1} underly the condition of having
no
contribution in one space and being completely reproducible in the
other (see Theorem \ref{thm_char_DG}), the core part of the pencil has to be adequate for both spaces,
$\mathbb{G}_{\eta + 1}(P)$
and $\mathbb{G}_{k - \eta}(P)$. This lucky circumstance determines the (block-symmetric) form of $C(\mathcal{L})$
completely as
depicted in the picture below and, no
matter how $\eta$ and $\epsilon$ are chosen, guarantees that $\Phi(
\mathcal{L}_{11}^{(i)}(\lambda)) = \alpha P(\lambda)$ always holds.
\vspace*{0.1cm}
\begin{center}
\footnotesize{ \begin{tikzpicture}[scale=0.5]
\matrix[matrix of math nodes, left delimiter={[}, right delimiter={]}]{ \node (P11) {\textcolor{black}{P_{\mu +
1}}}; & P_{\mu} &
P_{\mu - 1} & P_{\mu - 2} & \cdots & P_1 & \node (P1k) {P_0}; \\ \node (P21) {\textcolor{black}{P_{\mu}}}; &
\textcolor{black}{P_{\mu
- 1}} -
\lambda P_{\mu} & \textcolor{black}{P_{\mu - 2}} - \lambda P_{\mu - 1} & \textcolor{black}{P_{\mu - 3}} - \lambda
P_{\mu - 2} &
\cdots & \node (P2k-1) {\textcolor{black}{P_0} - \lambda P_1}; & - \lambda
P_0 \\ \textcolor{black}{P_{\mu - 1}} & P_{\mu - 2} - \lambda P_{\mu - 1} & P_{\mu - 3} - \lambda P_{\mu - 2} &
P_{\mu - 4} -
\lambda P_{\mu - 3} &  & \iddots & \\ \textcolor{black}{P_{\mu -2 }} & P_{\mu - 3} - \lambda P_{\mu - 2} & P_{\mu -
4} -
\lambda
P_{\mu - 3} & \cdots & \iddots & & \\ \vdots & \vdots & \vdots & \iddots &   &  & \\ \node (P1)
{\textcolor{black}{P_1}}; &
\textcolor{black}{P_0} - \lambda P_1 & - \lambda P_0 & & & & \\ \node (Pk1) {P_0}; &  - \lambda P_0 & &
& & & \\ };
\draw[dashed] (P11.west |- P1k.south) -- (P1k.south -| P1k.east);
\draw[dashed] (P11.east |- P1k.north) -- (P11.east |- Pk1.south);
\draw[color=black] (P21.south -| P21.west) -- (P21.south -| P2k-1.east);
\draw[color=black] (P21.south -| P2k-1.east) -- ++(0,2.1);
\draw[color=black] (P1.south -| P1.west) -- ++(6.5,0) -- ++ (0,8.4);
\fill[color=black!20!white, opacity=0.2] (P11.south -| P11.east) rectangle ++(22,-8.5);
\path (P11.south -| P11.east) -- ++(17,-7.8) node {core part $C(\mathcal{L})$};
\end{tikzpicture} }  \end{center}
\end{remark}

The next algorithm presents a procedure to reformulate a pencil from $\mathbb{DG}_{\eta + 1}(P)$ as an element of
$\mathbb{DG}_{\eta + i + 1}$ for all $i= 1, \ldots
, \lfloor \tfrac{\epsilon - \eta}{2} \rfloor$. This %exactly means
implies
$\mathcal{L}(\lambda)  \in \mathbb{G}_{\eta + i}(P)$ for all $i=2, \ldots , k-2 \eta - 1$.\\

\noindent \textbf{Algorithm 1: Shift-Procedure for Pencil Expressions} \\
Let $P(\lambda) = \sum_{i=0}^k P_i \lambda^i$ be an $n \times n$ matrix polynomial of degree $k= \epsilon + \eta +
1$ and assume
$\eta \leqslant \epsilon$. In addition, let a matrix pencil $\mathcal{L}(\lambda) \in \mathbb{DG}_{\eta + 1}(P)$ be
given as in
(\ref{doubleansatz_expr1}).  \\
\begin{itemize}
\item[1.]Choose any $i=1, \ldots , \lfloor \tfrac{\epsilon - \eta}{2}  \rfloor$ and partition
$\mathcal{H}_{\epsilon - \eta}(P)$ as follows:
$$  \mathcal{H}_{\epsilon - \eta}(P) = \left[ \begin{array}{c} J_i(P) \\ \hline \begin{array}{c|c|c} H_i(P) &
\mathcal{H}_{\epsilon
- \eta - 2i}(P) & 0 \\ \hline \mathcal{H}_i(P) & 0 & 0 \end{array} \end{array} \right].
$$
\item[2.] Partition $C_{11}$ as $C_{11}^{\mathcal{B}} = \big[ \, c_1 \; c_2 \; \ldots \; c_{\epsilon - \eta} \,
\big]$ with $c_i \in \mathbb{R}^{\eta n \times n}$ and compute the matrices
$$
 \widetilde{B}_{11}^{(i)} = \left[ \begin{array}{c|c} \begin{array}{c} B_{11} \\[0.1cm] \hline c_1^{\mathcal{B}} \\
\vdots \\
c_i^{\mathcal{B}} \end{array} &
0_{(\eta + i + 1)n \times i n} \end{array} \right] , \qquad
\widetilde{C}_{21}^{(i)} = \left[ \begin{array}{c|c}
\begin{array}{c} c_{(\epsilon - \eta)-i+1}^{\mathcal{B}} \\ \vdots \\ c_{\epsilon - \eta}^{\mathcal{B}} \end{array}
&
\mathcal{H}_i(P) \\[0.1cm] \hline C_{21} &
0_{\eta n \times i n}
\end{array} \right] ,
$$
and
$$  \widetilde{C}_{11}^{(i)} = \left[ \begin{array}{c|c} \begin{array}{c}
c_{i+1}^{\mathcal{B}} \\ \vdots \\ c_{(\epsilon - \eta)-i}^{\mathcal{B}} \end{array} & H_i(P) \end{array}
\right].
$$
Note $\widetilde{B}_{11}^{(i)} \in \mathbb{R}^{(\eta + 1 + i)n \times (\eta + i)n}$, $\widetilde{C}_{11}^{(i)} \in
\mathbb{R}^{(\epsilon - \eta - 2i)n \times (\eta + i)n}$ and $\widetilde{C}_{21}^{(i)} \in \mathbb{R}^{(\eta + i)n
\times (\eta +
i)n}$.

\item[3.] Compute the matrix $ \Omega_{\eta + i,P}(\lambda)$ from (\ref{matrix_Omega}) and express
$\mathcal{L}(\lambda)$ as
\begin{equation}
{\small
\left[ \begin{array}{c|c|c} I_{(\widetilde{\eta}+1)n} &  \widetilde{B}_{11}^{(i)} & 0
\\  \hline 0 &  \widetilde{C}_{11}^{(i)} &
\alpha \mathcal{H}_{\widetilde{\epsilon} - \widetilde{\eta}}(P) \\ \hline  0 &  \widetilde{C}_{21}^{(i)} & 0
\end{array} \right]
\left[
\begin{array}{c|c} \alpha \Omega_{\widetilde{\eta} ,
P}(\lambda)
& L_{\widetilde{\eta}}^T \otimes I_n
\\ \hline L_{\widetilde{\epsilon}} \otimes I_n & 0 \end{array} \right]
\left[ \begin{array}{c} \begin{array}{c|c} I_{(\widetilde{\epsilon}+1)n} & 0 \\ \hline
\widetilde{B}_2^{(i)} & \widetilde{C}_{2}^{(i)} \end{array}
\end{array} \right]} \label{shiftform}
\end{equation}
with $\widetilde{\eta} = \eta + i$, $\widetilde{\epsilon} = \epsilon - i$ and $ \begin{bmatrix}
\widetilde{B}_2^{(i)} &
\widetilde{C}_2^{(i)} \end{bmatrix} = \left[ \begin{array}{c} B_2 \hspace{2cm} C_2 \\
\hline
0_{in \times (\eta + 1)n} \; \, J_i(P) \; \, 0_{in \times \eta n} \end{array} \right]. $
\end{itemize}
Now the pencil $\mathcal{L}(\lambda)$ is an element of $\mathbb{DG}_{\widetilde{\eta} + 1}(P)$. Notice that we did
not formulate
$\mathcal{L}(\lambda)$ in terms of $\Pi_{\eta + i,P}^{\mathbb{DG}}(\lambda)$ as in (\ref{doubleansatz_expr1}).
Although this is
possible,
it is easier (and seems more natural) to just use $\Omega_{\eta + i,P}(\lambda)$ which is directly
available.\footnote{However,
having (\ref{shiftform}) we are certainly able to modify $\widetilde{B}_{11}^{(i)}$ and $\widetilde{B}_2^{(i)}$
appropriately to
express
$\mathcal{L}(\lambda)$ is the form (\ref{doubleansatz_expr1}).}

We illustrate this
procedure in the following example.
\begin{example}
Let $P(\lambda) = \sum_{i=0}^7 P_i \lambda^i$ be an $n \times n$ matrix polynomial of degree
$\textnormal{deg}(P) = 7.$
Consider the matrix pencil $\mathcal{L}(\lambda)$
 \footnotesize{ \[
\mathcal{L}(\lambda) = {\small \left[ \begin{array}{cccc:c:c:c} \lambda P_7 + P_6 & P_5 & -A & -B & -C & -D & -E \\
0 & P_4 &
\lambda
A + P_3 & \lambda B + P_2 & \lambda C + P_1 & \lambda D + P_0 & \lambda E \\ \hdashline -F & P_3 + \lambda F & P_2
- \lambda P_3
& P_1 - \lambda P_2 & P_0 - \lambda P_1 & - \lambda P_0 & 0 \\ \hdashline - G & P_2 + \lambda G & P_1 - \lambda P_2
& P_0 -
\lambda P_1 & - \lambda P_0 & & \\ \hdashline
- H & \lambda H + P_1 & P_0 - \lambda P_1 & - \lambda P_0 & & & \\ -J & P_0 + \lambda J & - \lambda P_0 & & & & \\
- K & \lambda
K & & & & & \end{array} \right]}
\]}
\normalsize{with arbitrary $n \times n$ matrices $A,B, \ldots , K.$}
\normalsize{This matrix pencil $\mathcal{L}(\lambda)$ is an element of $\mathbb{DG}_{2}(P)$ since it can be
expressed as}
\footnotesize{ \[
{ \left[ \begin{array}{cc|c|cccc} I_n &  & 0 &  &  &  &  \\  & I_n & 0 &
&  &  &  \\ \hline & & F & -P_3 & -P_2 &
-P_1 & -P_0 \\ & & G & -P_2 & -P_1 & -P_0 & \\  & & H & -P_1 & -P_0 & & \\ & &
J & -P_0 & & & \\ \hline  & & K & & & \end{array}
\right]
\Psi_1
%\left[ \begin{array}{c|c} \Omega_{1,P}(\lambda) & L_1^T \otimes I_n \\ \hline L_5 \otimes I_n & 0 \end{array}
%\right]
\left[
\begin{array}{cccccc|c} I_n & & & & & & \\ & I_n & & & & & \\ & & I_n & & & & \\ & & & I_n & & & \\ & & & & I_n & &
\\ & & & & &
I_n & \\ \hline 0 & 0 & A & B & C & D & E \end{array} \right] },
\] }
\normalsize{with}
\[\Psi_j =
\left[ \begin{array}{c|c} \Omega_{j,P}(\lambda) & L_j(\lambda)^T \otimes I_n \\ \hline L_{6-j}(\lambda) \otimes I_n & 0 \end{array}
\right].
\]
\normalsize{For $i=1$ we obtain according to Algorithm 2}
\footnotesize{ \[
\small { \left[ \begin{array}{ccc|cc|cc} I_n & & & 0 & 0 & & \\ & I_n & &
0 & 0 & & \\ & & I_n & F & 0 & & \\ \hline & & &
G &
-P_2 & -P_1 & -P_0 \\ & & & H & -P_1 & -P_0 & \\ \hline & & &
J & -P_0 & & \\ & & & K & 0 & & \end{array}
\right]
\Psi_2
\left[
\begin{array}{ccccc|cc}
I_n & & & & & & \\ & I_n & & & & & \\ & & I_n & & & & \\ & & & I_n & & & \\ & & & & I_n & & \\ \hline  0 & 0 & A &
B & C & D & E
\\ 0 & 0 & -P_3 & -P_2 & -P_1 & -P_0 &  \end{array} \right] .}
\]}
\normalsize{According to (\ref{shiftform}) this is the expression of $\mathcal{L}(\lambda)$ in the space
$\mathbb{DG}_{3}(P)$.
Now, since
$\lfloor \tfrac{\epsilon - \eta}{2} \rfloor = 2$ we may also consider the case $i=2$. Algorithm 2 gives in this
case}
\footnotesize{ \[
\footnotesize{ \left[ \begin{array}{cccc|ccc} I_n & & & & 0 & 0 &
0 \\ & I_n & & & 0 & 0 & 0 \\ & & I_n & &
F & 0 & 0 \\ & & & I_n & G &
0 & 0 \\
\hline  & & & & H & -P_1 & -P_0 \\ & & & & J
& -P_0 & 0 \\ & & & & K & 0 &
0 \end{array} \right]
\Psi_3
%\left[ \begin{array}{c|c} \Omega_{3,P}(\lambda) & L_3^T \otimes I_n \\ \hline L_3 \otimes I_n & 0 \end{array}
%\right]
\left[ \begin{array}{cccc|ccc} I_n & & & & & & \\ & I_n & &
&
& & \\ & & I_n & & & & \\ & & & I_n & & & \\ \hline 0 & 0 & A & B & C & D & E \\ 0 & 0 & -P_3 & -P_2 & -P_1 & -P_0
& \\ 0 & 0 &
-P_2
& -P_1 &
-P_0 & & \end{array} \right] .}
\] }
\normalsize{This is the expression of $\mathcal{L}(\lambda)$ as an element of $\mathbb{DG}_4(P)$. In
this case,
$\mathbb{DG}_4(P) = \mathbb{G}_4(P) \cap \mathbb{G}_4(P)$, so there are no additional restrictions for a pencil of
$\mathbb{G}_4(P)$ for being an element of $\mathbb{DG}_4(P)$. This complies with the disappearance of the
$\mathcal{H}$-block
and
the zero-blocks in (\ref{doubleansatz_expr1}).}
\end{example}
The following observation is immediate.
\begin{corollary}[Inclusion Property for $\mathbb{DG}_{\eta + 1}(P)$ Spaces] \label{cor_inclusionDG} \ \\
Let $P(\lambda)$ be an $n \times n$ matrix polynomial of degree $k= \epsilon + \eta + 1$.
Then we have
 \begin{equation}
  \mathbb{DG}_1(P) \, \subsetneqq \, \mathbb{DG}_2(P) \, \subsetneqq \, \cdots \, \subsetneqq \,
\mathbb{DG}_{\lceil
\tfrac{k}{2} \rceil}(P). \label{sequenceDG}
 \end{equation}
\end{corollary}

\subsection{Block-symmetric Pencils and the Spaces $\mathbb{BG}_{\eta + 1}(P)$}

This section is dedicated to the basic study of block-symmetric pencils in double block Kronecker ansatz spaces.
Block-symmetric block Kronecker pencils have already been considered in \cite{FassPS16}, whereas vector spaces of
block-symmetric
pencils are investigated in \cite{BueDFR15} and \cite{MacMMM06}. For motivation, consider once more the
matrix pencil $\mathcal{K}(\lambda)$ in (\ref{ex_DG2}).

\begin{remark} Example \ref{ex_blockLspace1} showed, in contrast to our experience
with the classical double ansatz space $\mathbb{DL}(P)$, that
not all matrix pencils in $\mathbb{DG}_{\eta + 1}(P)$ are block-symmetric.
Nevertheless, considering $\mathcal{K}(\lambda)$ from Example \ref{ex_blockLspace1} it is not hard to see how a block-symmetric matrix
pencil $\widetilde{\mathcal{K}}(\lambda)$ in $\mathbb{DG}_{2}(P)$ can be built. For $\widetilde{\mathcal{K}}(\lambda)$ we chose the $(1,1)$ block to be block-symmetric and adjust the bordering blocks to obtain a block-symmetric pencil:
     \begin{equation}
     \widetilde{\mathcal{K}}(\lambda) = \left[ \begin{array}{cc|ccc:c} \lambda P_6 + P_5
     & \tfrac{1}{2} P_4 & A & C & -B & -I_n \\ \tfrac{1}{2} P_4
     & P_3 & P_2 - \lambda A & P_1 - \lambda C &
     P_0 + \lambda B & \lambda I_n \\ \hdashline A & P_2 - \lambda A & P_1 -
     \lambda P_2 & P_0 - \lambda P_1 & - \lambda P_0 & 0 \\ C & P_1 - \lambda C & P_0 - \lambda P_1 & - \lambda
     P_0 & 0 & 0 \\ -B & P_0 + \lambda B & - \lambda P_0 & 0 & 0 & 0 \\ \hline -I_n & \lambda I_n & 0 & 0 & 0 &
     0 \end{array} \right]. \label{symmetrized_pencil}
     \end{equation} \end{remark}

\begin{definition}[Block-symmetric Block Kronecker Ansatz Space] \label{def_spaceBG} \ \\
Let $P(\lambda)$ be an $n \times n$ matrix polynomial of degree $k=\epsilon + \eta + 1$ and assume $\eta \leqslant
\epsilon$. Then we define
$$
 \mathbb{BG}_{\eta + 1}(P) = \big\lbrace \mathcal{L}(\lambda) \in \mathbb{DG}_{\eta + 1}(P) \; \big| \;
\mathcal{L}(\lambda) =
\mathcal{L}(\lambda)^{\mathcal{B}} \big\rbrace .
$$
\end{definition}

As Example \ref{ex_blockspace} immediately suggests, in general $\mathbb{DG}_{\eta + 1}(P) \neq \mathbb{BG}_{\eta +
1}(P)$ holds.
In fact, $\mathbb{BG}_{\eta + 1}(P)$ is a proper subspace of $\mathbb{DG}_{\eta + 1}(P)$ for $\eta > 0$ (see
Theorem
\ref{thm_char_BG} below) and therefore a nowhere dense subset in $\mathbb{DG}_{\eta + 1}(P)$. % (the case $\eta =
%0$ is

\begin{remark}
To find or construct block-symmetric pencils in $\mathbb{DG}_{\eta + 1}(P)$ several aspects have to be considered.
As in the previous discussion, the matrix pencils will be partitioned into a $3\times 3$ block matrix as in Figure
\ref{fig1}.
First and foremost (\ref{symmetrized_pencil}) reveals, that we have to take care of the bordering blocks in order
to
enforce pencils $\mathcal{L}(\lambda)$ in $\mathbb{DG}_{\eta + 1}(P)$ on being block-symmetric. Secondly, the upper
left square
diagonal block %of $\mathcal{L}_{11}(\lambda)$
certainly has to be block-symmetric as well. Thirdly, we do not have to take care of the core
part of the pencil which is, for pencils in $\mathbb{DG}_{\eta + 1}(P)$, block-symmetric anyway. These conditions
were taken into
account in the following algorithm.
\end{remark}

%%%%%% Construction of BG
%%%%%%%%%%%%%%%%%%%%%%%%%
\noindent \textbf{Algorithm 2: Construction Procedure for Block-symmetric Pencils} \\
\noindent Let $P(\lambda) = \sum_{i=0}^k P_i \lambda^i$ be an $n \times n$ matrix polynomial of degree $k= \epsilon
+ \eta + 1$.
\\
\begin{itemize}
\item[1.] Compute the matrix
$$
\Sigma^{\mathbb{BG}}_{\eta,P}(\lambda) = \begin{bmatrix}
\lambda P_k + P_{k-1} & & & \\
& \lambda P_{k-2} + P_{k-3} & & \\ & & \ddots & \\ & & & \lambda P_{\epsilon - \eta +1} + P_{\epsilon - \eta}
\end{bmatrix} $$
and set $ \Pi_{\eta ,
P}^{\mathbb{BG}}(\lambda) = \begin{bmatrix}  \Sigma_{\eta , P}^{\mathbb{BG}}(\lambda) & \mathcal{R}_{\eta,P}
\end{bmatrix}.$ Note that $\Sigma_{\eta, P}^{\mathbb{BG}}(\lambda) \in \mathbb{R}[\lambda]^{(\eta + 1)n \times (\eta + 1)n}.$ (For the definition of $\mathcal{R}_{\eta , P}$ see (\ref{Rmatrix})). %\ps{Recall that
%$$ \mathcal{R}_{\eta,P} =
%\left[\begin{array}{c} 0_{\eta n \times (\epsilon -
%\eta)n} \\
%\hline \begin{array}{ccc} P_{\epsilon - \eta - 1} & \cdots & P_0 \end{array} \end{array} \right] \in
%\mathbb{R}^{(\eta + 1)n
%\times (\epsilon - \eta)n}.$$
%Note that
%$\Sigma_{\eta, P}^{\mathbb{BG}}(\lambda) \in \mathbb{R}[\lambda]^{(\eta + 1)n \times (\eta + 1)n}.$

\item[2.] Compute the matrix $$ C_1 = \left[
\begin{array}{c|c} C_{11} & \alpha \mathcal{H}_{\epsilon - \eta}(P) \\[0.1cm] \hline C_{21} & 0_{\eta n \times
(\epsilon -
\eta)n} \end{array} \right] \in \mathbb{R}^{ \epsilon n \times \epsilon n} $$
with arbitrary matrices $C_{11} \in \mathbb{R}^{(\epsilon - \eta)n \times \eta n}$ and $C_{21} \in \mathbb{R}^{\eta
n \times \eta
n}$.

\item[3.] Choose an arbitrary matrix $B_{11} \in \mathbb{R}^{(\eta + 1)n \times \eta n}$ and set
\begin{equation} B_1= \begin{bmatrix}  B_{11} & 0_{(\eta
+1)n \times (\epsilon - \eta)n} \end{bmatrix} \qquad C_2 = C_{21}^{\mathcal{B}} \qquad B_2 = \begin{bmatrix}
B_{11}^{\mathcal{B}}
& C_{11}^{\mathcal{B}} \end{bmatrix}. \label{help1} \end{equation}

\item[4.]
Construct the $kn \times kn$ matrix pencil $\mathcal{L}(\lambda) \in \mathbb{DG}_{\eta + 1}(P)$:
\begin{align} \mathcal{L}(\lambda) &= \left[ \begin{array}{c|c} I_{(\eta + 1)n} & B_1 \\ \hline 0 & C_1 \end{array}
\right] \left[ \begin{array}{c|c} \alpha \Pi_{\eta, P}^{\mathbb{BG}}(\lambda) & L_{\eta}^T \otimes I_n
\\ \hline L_{\epsilon} \otimes I_n & 0 \end{array} \right] \left[ \begin{array}{c|c} I_{(\epsilon +1)n} & 0 \\
\hline B_2 & C_2
\end{array} \right].  \label{block_constr}
\end{align}
\end{itemize}
The matrix pencil $\mathcal{L}(\lambda)$ is explicitly given as
\footnotesize{ $$
\mathcal{L}(\lambda) = \left[ \begin{array}{c|c|c} I_{(\eta + 1)n} & B_{11} & 0 \\[0.1cm]  \hline 0 & C_{11} &
\alpha \mathcal{H}_{\epsilon -
\eta} \\ \hline 0 & C_{21} & 0  \end{array} \right]  \left[ \begin{array}{c|c} \alpha
\Pi_{\eta, P}^{\mathbb{BG}}(\lambda)
& L_{\eta}^T \otimes I_n
\\ \hline L_{\epsilon} \otimes I_n & 0 \end{array} \right]
\left[ \begin{array}{c|c} I_{(\epsilon + 1)n} & 0 \\[0.1cm] \hline
\begin{array}{c|c} B_{11}^{\mathcal{B}} & C_{11}^{\mathcal{B}} \end{array}  & C_{21}^{\mathcal{B}}
\end{array} \right].
$$ }

\normalsize{Since $\Sigma_{\eta,P}^{\mathbb{BG}}(\lambda)$ is block-symmetric by construction,
(\ref{help1}) ensures the block-symmetry of
$\mathcal{L}(\lambda)$ in total. To this, remember that the core part of a pencil in $\mathbb{DG}_{\eta + 1}(P)$ is
always
block-symmetric.
It is easily seen that the conditions (\ref{help1}) are not only sufficient, but also necessary for
$\mathcal{L}(\lambda)$ in
(\ref{block_constr}) to be block-symmetric (recall (\ref{symmetrized_pencil}) and (\ref{ex_DG2})).}

\begin{theorem}[Characterization of $\mathbb{BG}_{\eta + 1}(P)$] \label{thm_char_BG} \ \\
Let $P(\lambda)$ be an $n \times n$ matrix polynomial of degree $k=\epsilon + \eta + 1$ and assume $\eta \leqslant
\epsilon$. Then $\mathbb{BG}_{\eta + 1}(P)$ is a vector space over $\mathbb{R}$ having dimension
$$
 \textnormal{dim} \big( \mathbb{BG}_{\eta + 1}(P) \big) = k \eta n^2 + 1.
$$
Any matrix pencil $\mathcal{L}(\lambda) \in \mathbb{BG}_{\eta + 1}(P)$ may be characterized as
\begin{equation} \footnotesize{   \mathcal{L}(\lambda)= \left[ \begin{array}{c|c|c} I_{(\eta + 1)n} &  B_{11} & 0
\\  \hline 0 &
C_{11}
& \alpha \mathcal{H}_{\epsilon -
\eta} \\ \hline 0 & C_{21} & 0 \end{array} \right]  \left[ \begin{array}{c|c} \alpha
\Pi_{\eta , P}^{\mathbb{BG}}(\lambda)
& L_{\eta}^T \otimes I_n
\\ \hline L_{\epsilon} \otimes I_n & 0 \end{array} \right]
\left[ \begin{array}{c|c} I_{(\epsilon + 1)n} & 0 \\[0.1cm] \hline
\begin{array}{c|c} B_{11}^{\mathcal{B}} & C_{11}^{\mathcal{B}} \end{array} & C_{21}^{\mathcal{B}}
\end{array} \right] } \label{blockansatz_expr1}
\end{equation} \normalsize{with arbitrary matrices $B_{11} \in \mathbb{R}^{(\eta + 1)n \times \eta n}$, $C_{11} \in
\mathbb{R}^{(\epsilon - \eta)n \times
\eta n}, C_{21} \in \mathbb{R}^{\eta n \times \eta n}$ and $\alpha \in \mathbb{R}$. Moreover, unless $\eta =0$,
$\mathbb{BG}_{\eta + 1}(P)$ is a proper subspace of both $\mathbb{DG}_{\eta + 1}(P)$ and $\mathbb{DG}_{k -
\eta}(P)$. }
\end{theorem}

The next results about $\mathbb{BG}_{\eta + 1}(P)$ are immediate consequences of Theorem \ref{master4} and
Corollary
\ref{cor_inclusionDG}.

\begin{corollary}[Linearization Condition for $\mathbb{BG}_{\eta + 1}(P)$] \label{master5} \ \\
Let $P(\lambda)$  be a square and regular matrix polynomial of degree $k= \eta + \epsilon + 1$.
Let $\mathcal{L}(\lambda) \in \mathbb{BG}_{\eta + 1}(P)$ be given in the form (\ref{blockansatz_expr1}).
Assume $\epsilon \neq \eta$. Then the
following statements are equivalent:
\begin{enumerate}
 \item $\mathcal{L}(\lambda)$ is a strong linearization for $P(\lambda)$.
 \item $P_0 \in \textnormal{GL}_n( \mathbb{R}), C_{21} \in \textnormal{GL}_{\eta n}(\mathbb{R})$ and $\alpha \in
\mathbb{R}
\setminus \lbrace 0 \rbrace$.
\end{enumerate}
\end{corollary}
For $\epsilon = \eta$ the equivalence in Corollary \ref{master5} holds without the condition $P_0 \in \textnormal{GL}_n( \mathbb{R})$ in the second statement (due to the disappearance of the $\mathcal{H}$-block).
In this case, the implication $2. \Rightarrow 1.$ holds also for singular matrix polynomials according to Theorem
\ref{thm_lincondition}.
Moreover, certainly  Corollary \ref{generic1} still holds. That is, whenever zero is not an eigenvalue of
$P(\lambda)$, i.e.,
$P_0 \in
\textnormal{GL}_n( \mathbb{R})$, almost every matrix pencil in $\mathbb{BG}_{\eta + 1}(P)$ is a strong
linearization for
$P(\lambda)$ regardless whether $P(\lambda)$ is regular or singular. Moreover, the inclusion property from the
previous section
is still valid for block-symmetric pencils.

\begin{lemma}[Inclusion Property for $\mathbb{BG}_{\eta + 1}(P)$ Spaces] \ \\
Let $P(\lambda)$ be an $n \times n$ matrix polynomial of degree $k= \epsilon + \eta + 1$.
Then we have
 \begin{equation}
  \mathbb{BG}_1(P) \, \subsetneqq \, \mathbb{BG}_2(P) \, \subsetneqq \, \cdots \, \subsetneqq \,
\mathbb{BG}_{\lceil
\tfrac{k}{2} \rceil}(P). \label{sequenceBG}
 \end{equation}
\end{lemma}

To illustrate the construction procedure from Algorithm 2 consider the following simple example.

\begin{example} \label{ex_blocksymmpencils}
 Let $P(\lambda) = \sum_{i=0}^7 P_i \lambda ^i$ be an $n \times n$ matrix polynomial of degree
$\textnormal{deg}(P)=7$.
 First consider the case $\eta = 1$ and $\epsilon = k - \eta - 1 = 5$. The construction procedure easily gives
 $$ \mathcal{H}_4 = \begin{bmatrix}-P_3 & -P_2 & - P_1 & -P_0 \\ -P_2 & -P_1 & -P_0 & \\ -P_1 & -P_0 & & \\ -P_0 &
& &
 \end{bmatrix} \in \mathbb{R}^{4 n \times 4n}$$
 and $\Pi^{\mathbb{BG}}_{1, P}(\lambda) = \begin{bmatrix} \lambda P_7 + P_6 & 0_n & 0_n & 0_n & 0_n &
0_n \\  0_n & \lambda P_5 + P_4 & P_3 & P_2 & P_1 & P_0 \end{bmatrix} $. Choose $B_{11} = 0$, $C_{11} = 0$ and
$C_{21} = I_n$.
Then computing $\mathcal{L}(\lambda)$ from (\ref{block_constr}) with $\alpha = 1$ yields
$$ \small{ \mathcal{L}(\lambda) = \left[ \begin{array}{cc|cccc:c} \lambda P_7 + P_6 & 0 & 0 &
0 & 0 & 0 &
-I_n \\
0 & \lambda P_5 + P_4 &
P_3 & P_2 & P_1 & P_0 & \lambda I_n \\ \hdashline 0 & P_3 & P_2 - \lambda P_3 & P_1 - \lambda P_2 & P_0 - \lambda
P_1 & - \lambda P_0 & 0 \\ 0 & P_2 & P_1 - \lambda P_2 & P_0 - \lambda P_1 & - \lambda P_0 & 0 & 0\\ 0 & P_1 & P_0
- \lambda P_1
& -
\lambda P_0 & 0 & 0 & 0 \\ 0 & P_0 & - \lambda P_0 & 0 & 0 & 0  & 0 \\ \hline -I_n & \lambda I_n & 0& 0 & 0 & 0 & 0
\end{array}
\right]}
$$
which is indeed a block-symmetric $7n \times 7n$ matrix pencil. Thus $\mathcal{L}(\lambda) \in \mathbb{BG}_2(P)$.
Note
that the choice of $B_{11}$ and $C_{11}$ has no influence on $\mathcal{L}(\lambda)$ for being a linearization. In
fact,
the nonsingularity of $P_0$ and $C_{21}$ is the decisive factor, while choosing $B_{11}$ and $C_{11}$ to be singular
matrices does
not affect the linearization property of $\mathcal{L}(\lambda)$ at all.

Now consider $\eta = 2$ and $\epsilon = k - \eta - 1=4$. Then
$$ \mathcal{H}_2 = \begin{bmatrix} - P_1 & -P_0 \\ -P_0 & 0 \end{bmatrix} \in \mathbb{R}^{2n \times 2n}.$$
Now choose $C_{11}= \begin{bmatrix} -P_7 & -P_6 \\ -P_5 & -P_4 \end{bmatrix}$ and $C_{21} = \begin{bmatrix} -P_3 &
-P_2 \\ -P_1 &
-P_0 \end{bmatrix}$. The computation in (\ref{block_constr}) gives
\footnotesize{ $$  \mathcal{K}(\lambda) = \left[ \begin{array}{ccc|cc:cc} \lambda P_7 + P_6 & 0 & 0
&
P_7 & P_5 &
P_3 & P_1 \\
                           0 & \lambda P_5 + P_4 & 0 & P_6 - \lambda P_7 & P_4 - \lambda P_5 & P_2
- \lambda P_3
& P_0 - \lambda P_1 \\ 0 & 0 & \lambda P_3 + P_2 & P_1 - \lambda P_6 & P_0 - \lambda P_4 & - \lambda P_2 & -
\lambda P_0
\\ \hdashline
P_7 & P_6 - \lambda P_7 & P_1 - \lambda P_6 & P_0 - \lambda P_1 & - \lambda P_0 & 0 & 0 \\
P_5 & P_4 - \lambda P_5 & P_0 - \lambda P_4 & - \lambda P_0 & 0 & 0 & 0 \\ \hline
P_3 & P_2 - \lambda P_3 & - \lambda P_2 & 0 & 0 & 0 & 0 \\
P_1 & P_0 - \lambda P_1 & - \lambda P_0 & 0 & 0 & 0 & 0 \end{array} \right]  $$ }
\normalsize{which is block-symmetric. Therefore we have $\mathcal{K}(\lambda) \in \mathbb{BG}_3(P)$.}
\end{example}

\begin{remark}
Consider  $\mathcal{L}(\lambda)$ and $\mathcal{K}(\lambda)$  from the last example.
 $\mathcal{L}(\lambda)$ is a strong linearization for $P(\lambda)$ if and only if $\textnormal{det}(P_0) \neq 0,$
whereas $\mathcal{K}(\lambda)$ is a strong linearization for $P(\lambda)$ if and only if $\textnormal{det}(P_0),
\textnormal{det}(P_1), \textnormal{det}(P_2) \neq 0$ (see
Theorem \ref{thm_lincondition}). Neither the classical ansatz space approach (see \cite{MacMMM06}) nor the pure
block Kronecker pencils from \cite{DopLPVD16} cover block-symmetric pencils like these.
\end{remark}

%%%%%%%%%%%%%%%%%
%%%%%% Section
%%%%%%%%%%%%%%%%%%%%%%

\section{Block Kronecker Ansatz Spaces and the Classical Ansatz Spaces}
\label{sec:L1L2}

As this was pointed out before, there is a strong connection between the classical ansatz spaces $\mathbb{L}_1(P), \mathbb{L}_2(P)$ and $\mathbb{DL}(P)$ and the block Kronecker ansatz spaces introduced in this paper. This section is devoted to the establishment of this connection.

Let $P(\lambda)$ be an $n \times n$ matrix polynomial of degree $k$. For $\eta = 0$ the ansatz equation (\ref{ansatzequation1})
has the form
$$ \mathcal{L}(\lambda) \big( \Lambda_{k-1}(\lambda) \otimes I_n) = \alpha e_1 \otimes P(\lambda)$$
which coincides with the ansatz equation for $\mathbb{L}_1(P)$ (see (3.4) in \cite{MacMMM06}) for the choice $v = \alpha e_1$. According to Theorem \ref{thm_generalspace} %we know that
every matrix pencil
$\mathcal{L}(\lambda)$ in $\mathbb{G}_1(P)$ may be expressed as
$$  \mathcal{L}(\lambda) = \left[ \begin{array}{c|c} I_n & B_1 \\ \hline 0 & C_1 \end{array} \right]
\left[
\begin{array}{c} \alpha \Sigma_{0,P}(\lambda) \\ \hline  L_{k-1}(\lambda) \otimes I_n \end{array} \right] = \left[
\begin{array}{c|c} I_n & B_1 \\ \hline 0 & C_1 \end{array} \right] \mathcal{F}_{\alpha, 0, P}(\lambda).
$$
Multiplying $\mathcal{L}(\lambda)$ from the left with
$$ \mathcal{V}_{\textnormal{left}} = \left[ \begin{array}{c|c} v \otimes I_n & \begin{array}{c} 0_{n \times
\epsilon n} \\ \hline
I_{\epsilon n} \end{array} \end{array} \right] \in \mathbb{R}^{(\epsilon + 1)n \times (\epsilon + 1)n}$$
 gives a pencil that satisfies $\mathcal{L}(\lambda) ( \Lambda_{k-1}(\lambda) \otimes I_n) = v \otimes P(\lambda)$  (due to the multiplication with $v \in \mathbb{R}^k$, the scalar $\alpha \in \mathbb{R}$ is ignored until
further
notice, i.e. we set $\alpha = 1$). On the other hand it is easily seen, that any matrix pencil of the form
\begin{equation}  \mathcal{L}(\lambda) = \left[ \begin{array}{c|c} v \otimes I_n & \begin{array}{c} B_1
\\ \hline
C_1 \end{array} \end{array} \right]  \mathcal{F}_{1,0,P}(\lambda) = \left[ \begin{array}{c|c} v \otimes I_n & \begin{array}{c} B_1
\\ \hline
C_1 \end{array} \end{array} \right]  \textnormal{Frob}_P(\lambda) \label{L1_compactform} \end{equation}
satisfies $\mathcal{L}(\lambda)( \Lambda_{k-1}(\lambda) \otimes I_n) = v \otimes P(\lambda)$. Now, verifying
that (\ref{L1_compactform}) is essentially just a reformulation of \cite[Thm. 3.5]{MacMMM06}, we have derived an
equivalent, but alternative description of $\mathbb{L}_1(P)$. In the context of orthogonal bases, this result was
already
obtained in \cite{FassS16}.

\begin{corollary}[Characterization of $\mathbb{L}_1(P)$] \label{cor_spaceL1} \ \\
 Let $P(\lambda)$ be an $n \times n$ matrix polynomial of degree $k$.
Then $\mathcal{L}(\lambda)$ satisfies the
classical ansatz equation $\mathcal{L}(\lambda)( \Lambda_{k-1} \otimes I_n) = v \otimes P(\lambda)$ if and
only if
\begin{equation}
 \mathcal{L}(\lambda) = \big[ \, v \otimes I_n \; \mathcal{Z} \, \big] \textnormal{Frob}_P(\lambda)
\qquad \mathcal{Z}= \left[ \begin{array}{c} B_1 \\ \hline C_1 \end{array} \right]
\label{L1}
\end{equation}
for some arbitrary matrix $\mathcal{Z} \in \mathbb{R}^{kn \times (k-1)n}$.
\end{corollary}

The characterization in (\ref{L1}) together with Theorem \ref{thm_lincondition} yields a very simple linearization
condition
for pencils in $\mathbb{L}_1(P)$ for regular matrix polynomials $P(\lambda)$ that is equivalent to but different from the well known $Z$-rank condition (see \cite[Cor. 2]{FassS16}).
\begin{corollary}\label{lemma_neu}
A matrix pencil $\mathcal{L}(\lambda) \in
\mathbb{L}_1(P)$ as in (\ref{L1}) is a strong linearization for a regular $P(\lambda) = \sum_{i=0}^k P_i\lambda^i \in
\mathbb{R}[\lambda]^{n \times n}$  with $P_k \neq 0$ if and only if $ [ \, v \otimes I_n \; \mathcal{Z} \, ]$ is a
nonsingular matrix, i.e. $\textnormal{rank}([ \, v \otimes I_n \; \mathcal{Z} \, ])=kn.$
\end{corollary}
In this case, the eigenvectors of $\mathcal{L}(\lambda)$ are exactly the eigenvectors of
$\textnormal{Frob}_P(\lambda)$ (see \cite[Thm. 3.8]{MacMMM06}).

A similar characterization of $\mathbb{L}_2(P)$ can be derived in an analogous way \cite[Thm. 2]{FassS16}.
Therefore, we
obtain that
$\mathbb{L}_2(P)$ consists of all matrix pencils $\mathcal{L}(\lambda)$ having the form
\begin{equation}
 \mathcal{L}(\lambda) = \text{Frob}^{\mathcal{B}}_P(\lambda) \left[ \begin{array}{c} v^T \otimes I_n \\
\hline \mathcal{Z}
\end{array} \right] \label{L2}
\end{equation}
for some arbitrary matrix $\mathcal{Z} = [ \, B_1 \; | \; C_1 \, ] \in \mathbb{R}^{(k-1)n \times kn}$.
These matrix
pencils satisfy the (second) classical ansatz equation $(\Lambda_{k-1}^T \otimes I_n) \mathcal{L}(\lambda) =
v^T \otimes
P(\lambda)$.
 Similar as before, (\ref{L2}) can be seen as
a reformulation of \cite[Lemma 3.11]{MacMMM06} and we
obtain statements analogous to Corollaries \ref{cor_spaceL1} and \ref{lemma_neu}.

The ansatz space $\mathbb{DL}(P)$ was introduced in \cite{MacMMM06} as the intersection of $\mathbb{L}_1(P)$ and $\mathbb{L}_2(P)$.
As the final result of this section we state the following lemma that connects the three kinds of ansatz spaces introduced in this paper and the
$\mathbb{DL}(P)$ space.

\begin{lemma}
Let $P(\lambda)$ be an $n \times n$ matrix polynomial of degree $k \geq 2$. Then
$$ \bigcap_{\eta = 0}^{k-1} \mathbb{G}_{\eta + 1}(P) = \bigcap_{\eta = 0}^{k-1} \mathbb{DG}_{\eta + 1}(P) = \bigcap_{\eta = 0}^{k-1} \mathbb{BG}_{\eta + 1}(P) = \mathbb{DL}(P)|_{\langle e_1 \rangle}.$$
Here $\langle e_1 \rangle$ denotes the one-dimensional subspace of $\mathbb{R}^k$ spanned by $e_1$.
\end{lemma}

\begin{proof}
Since $\mathbb{G}_1(P) \cap \mathbb{G}_k(P) = \mathbb{DL}(P)|_{\langle e_1 \rangle}$ the lemma follows from the observations in (\ref{sequenceDG}) and (\ref{sequenceBG}).
\end{proof}

Corollary \ref{cor_spaceL1} has particularly nice consequences for the ansatz spaces $\mathbb{L}_1(P),$ $\mathbb{L}_2(P)$ and $\mathbb{DL}(P)$. In fact, many well-known results on $\mathbb{L}_1(P)$ admit easily accessible proofs considering the form (\ref{L1}) instead of \cite[Thm. 3.5]{MacMMM06} (see \cite{FassS16}).
In the next section we show that the standard basis of $\mathbb{DL}(P)$, i.e. the rectangular matrices $\mathcal{Z}_i$ corresponding to the basis pencils $$\mathcal{B}_i(\lambda) = \big[ \, e_i \otimes I_n \; \mathcal{Z}_i \, \big] \textnormal{Frob}_P(\lambda) \in \mathbb{DL}(P)  \quad i=1, \ldots , k$$ can in fact be immediately determined from a tableau containing the matrix coefficients of $P(\lambda)$ without any computation at all.

\subsection{Application: Computing the Standard Basis of $\mathbb{DL}(P)$}

Consider the double ansatz space  $\mathbb{DL}(P) = \mathbb{L}_1(P) \cap \mathbb{L}_2(P)$ (\ref{DLP}). Any matrix
pencil
$\mathcal{L}(\lambda) \in \mathbb{DL}(P)$ is blocksymmetric \cite[Theorem 3.4]{HigMMT06}. In \cite[Section
3.3]{HigMMT06} it is
discussed how to compute the \enquote{standard basis pencils} in $\mathbb{DL}(P)$ corresponding to the standard
basis $\{e_1,
\ldots, e_k\} \in \mathbb{R}^k.$ Certainly, computing the standard basis for $\mathbb{DL}(P)$, see \cite[Sec.
3.3]{HigMMT06},
for $\mathbb{DL}(P)$ from \cite[Theorem 3.5]{HigMMT06} seems not to be a complicated task.
However, regarding the expression (\ref{L1}) for matrix pencils in $\mathbb{L}_1(P)$, computing a particular
blocksymmetric pencil
$\mathcal{L}(\lambda) \in \mathbb{L}_1(P)$ for some given ansatz vector $v \in \mathbb{R}^k$ breaks down to the
computation of
the corresponding matrix $\mathcal{Z} \in
\mathbb{R}^{kn \times (k-1)n}$.
Thus, computing
$\mathcal{Z}_j$ for $\mathcal{B}_j := [ \, (e_j \otimes I_n) \; \mathcal{Z}_j \, ] \textnormal{Frob}_P(\lambda)
\in
\mathbb{DL}(P)$ seems
even simpler and does only require the computation of one $kn \times (k-1)n$ matrix instead of the set-up of two
$kn \times kn$
matrices. In fact in was shown in \cite{FassS16} that $\mathcal{Z}$ has some blocksymmetric structure, too.

To this end, let $P(\lambda)=\sum_{i=0}^k P_i
\lambda^k$ be a square matrix polynomial of degree $k$.
Using the matrix coefficients of $P(\lambda)$ we define the $\mathcal{P}$-Tableau corresponding to $P(\lambda)$ as
in Figure
\ref{P-tableau}.
\begin{figure}
\begin{center}
\begin{tikzpicture}[scale=1]
 \draw[help lines] (0,0) grid (8,5);
 \draw[line width=0.4mm] (0,0) rectangle (8,5);
\node at (0.5,0.5) {\footnotesize{$P_{k-1}$}};
\node at (1.5,0.5) {\footnotesize{$P_{k-2}$}};
\node at (2.5,0.5) {\footnotesize{$\hdots $}};
\node at (3.5,0.5) {\footnotesize{$P_1$}};
\node at (4.5,0.5) {\footnotesize{$-P_{0}$}};
\node at (0.5,1.5) {\footnotesize{$P_{k}$}};
\node at (1.5,1.5) {\footnotesize{$P_{k-1}$}};
\node at (2.5,1.5) {\footnotesize{$\hdots$}};
\node at (3.5,1.5) {\footnotesize{$P_{2}$}};
\node at (4.5,1.5) {\footnotesize{$-P_1$}};
\node at (5.5,1.5) {\footnotesize{$-P_0$}};
\node at (1.5,2.5) {\footnotesize{$\iddots$}};
\node at (2.5,2.5) {\footnotesize{$\iddots$}};
\node at (3.5,2.5) {\footnotesize{$\vdots$}};
\node at (4.5,2.5) {\footnotesize{$\vdots$}};
\node at (5.5,2.5) {\footnotesize{$\iddots$}};
\node at (6.5,2.5) {\footnotesize{$\iddots$}};
\node at (2.5,3.5) {\footnotesize{$P_k$}};
\node at (3.5,3.5) {\footnotesize{$P_{k-1}$}};
\node at (4.5,3.5) {\footnotesize{$-P_{k-2}$}};
\node at (5.5,3.5) {\footnotesize{$\hdots$}};
\node at (6.5,3.5) {\footnotesize{$-P_1$}};
\node at (7.5,3.5) {\footnotesize{$-P_0$}};
\node at (3.5,4.5) {\footnotesize{$P_k$}};
\node at (4.5,4.5) {\footnotesize{$0$}};
\node at (5.5,4.5) {\footnotesize{$\hdots$}};
\node at (6.5,4.5) {\footnotesize{$\hdots$}};
\node at (7.5,4.5) {\footnotesize{$0$}};
\draw[line width=0.4mm] (4,0) -- (4,5);
\node at (4,5.5) {$\mathcal{P}$-Tableau};
\end{tikzpicture}
\end{center}
\caption{$\mathcal{P}$-tableau corresponding to $P(\lambda) = \sum_{i=0}^k P_i \lambda^i$.}
\label{P-tableau}
\end{figure}
Now the matrices $\mathcal{Z}_i$ that correspond to a blocksymmetric matrix pencil $\mathcal{L}(\lambda) =
\mathcal{B}_j(\lambda)   \in \mathbb{L}_1(P)$
having the form (\ref{L1}) with ansatz vector $e_j$ can easily be determined by the tableau.
Therefore, regard the tableau as a $k \times 2(k-1)$ matrix and denote the left half by $\mathcal{J}_P$ and the
right half by
$\mathcal{H}_P$.

\begin{lemma} \label{lem_standardbasis}
 Let $P(\lambda) = \sum_{i=0}^k P_i \lambda^k$ be a square matrix polynomial of degree $k$ and
$\mathcal{L}(\lambda) \in
\mathbb{L}_1(P)$ with ansatz vector $v=e_i$. Then $\mathcal{L}(\lambda) = [\, (e_i \otimes I_n) \; \mathcal{Z}_i \,
]
\textnormal{Frob}_P(\lambda) \in \mathbb{L}_2(P)$ if and only if
\begin{equation}
 \mathcal{Z}_i =
\left\{
\begin{array}{ll}
\mathcal{H}_P & i = 1\\
\mathcal{J}_P(1:i,k-i:k-1) \oplus \mathcal{H}(i+1:k,1:i+1) & 1 < i < k\\
\mathcal{J}_P & i = k
\end{array}
\right.
 \label{formula_blocksymm}
\end{equation}
\end{lemma}

\begin{proof}
First observe that any matrix pencil $\mathcal{L}(\lambda) \in \mathbb{L}_1(P)$ may be expressed as
\begin{align} \mathcal{L}(\lambda) &= \big[ \, (v \otimes I_n) \; \mathcal{Z} \, \big] \textnormal{Frob}_P(\lambda)
\notag \\ &=
v \otimes \Sigma_{0,P}(\lambda) + \mathcal{Z}L_{k-1}(\lambda) \notag \\
 &= \big[ \, (v \otimes P_k) \; \mathcal{Z} \, \big] \lambda + \big( v \otimes \Sigma_{0,P}(0) + \mathcal{Z}
L_{k-1}(0) \big)
\label{pencil_L1L0}
\end{align}
Now notice that (\ref{pencil_L1L0}) expresses $\mathcal{L}(\lambda)$ in the form $\mathcal{L}(\lambda) =
\mathcal{L}_1 \lambda +
\mathcal{L}_0$ with two $kn \times kn$ matrices $\mathcal{L}_1$ and $\mathcal{L}_0$. This form was mainly
considered in
\cite{HigMMT06}.
Comparing $X_m$ from \cite[(3.8a)]{HigMMT06} with $\mathcal{Z}_m$ as defined in Lemma \ref{lem_standardbasis} and
considering \cite[Thm. 3.5]{HigMMT06} shows the statement.
\end{proof}

To illustrate Lemma \ref{lem_standardbasis} consider the following examples. Deviating from our notation, the
polynomial
coefficients in the example below are denoted $A,B, C, \ldots$ to highlight the similarity to \cite[Table
1/2]{MacMMM06}
and \cite[Table 3.1/3.2]{HigMMT06}.

\begin{example}
Let $P(\lambda)= A \lambda^2 + B \lambda + C$ be an $n \times n$ matrix polynomial of degree
$\textnormal{deg}(P(\lambda))=2$.
Then the matrix $\mathcal{Z}$ in (\ref{L1}) has dimension $2n \times n$. Therefore, the
$\mathcal{P}$-tableau has
dimension $2n \times 2n$ and is easily computed as
\begin{center}
\begin{tikzpicture}[scale=0.8]
\draw[help lines] (0,0) grid (2,2);
\node at (0.5,0.5) {\footnotesize{$B$}};
\node at (1.5,0.5) {\footnotesize{$-C$}};
\node at (0.5,1.5) {\footnotesize{$A$}};
\node at (1.5,1.5) {\footnotesize{$0$}};
%\node at (-1.5,1) {$\mathcal{P}$-Tableau};
\draw[line width=0.4mm] (1,0) -- (1,2);
\end{tikzpicture}
\end{center}
and we have $\mathcal{Z}_1 = \mathcal{H}_P$ and $\mathcal{Z}_2 = \mathcal{J}_P$.
Now let $P(\lambda) = A \lambda^3 + B \lambda ^2 + C \lambda + D$ be of degree $\textnormal{deg}(P(\lambda))=3$.
Then the matrix
$\mathcal{Z}$ in (\ref{L1}) has dimension $3n \times 2n$ and the $\mathcal{P}$-tableau
dimension $3n \times 4n$. It is given by
 \begin{center}
 \begin{tikzpicture}
  \draw[help lines] (0,0) grid (4,3);
  \node at (0.5, 0.5) {\footnotesize{$B$}};
  \node at (1.5, 0.5) {\footnotesize{$C$}};
  \node at (2.5, 0.5) {\footnotesize{$-D$}};
  \node at (3.5, 0.5) {\footnotesize{$0$}};
  \node at (0.5, 1.5) {\footnotesize{$A$}};
  \node at (1.5, 1.5) {\footnotesize{$B$}};
  \node at (2.5, 1.5) {\footnotesize{$-C$}};
  \node at (3.5, 1.5) {\footnotesize{$-D$}};
  \node at (0.5, 2.5) {\footnotesize{$0$}};
  \node at (1.5, 2.5) {\footnotesize{$A$}};
  \node at (2.5, 2.5) {\footnotesize{$0$}};
  \node at (3.5, 2.5) {\footnotesize{$0$}};
  \draw[line width=0.4mm] (2,0) -- (2,3);
  %\node at (-1.5,1.5) {$\mathcal{P}$-Tableau};
 \end{tikzpicture}
 \end{center}
 The three structures according to formula (\ref{formula_blocksymm}) are
  \begin{center}
  %%%%% Bild 1
  \begin{tikzpicture}[scale=0.8]
  \draw[help lines] (0,0) grid (4,3);
  \node at (0.5, 0.5) {\footnotesize{$B$}};
  \node at (1.5, 0.5) {\footnotesize{$C$}};
  \node at (2.5, 0.5) {\footnotesize{$-D$}};
  \node at (3.5, 0.5) {\footnotesize{$0$}};
  \node at (0.5, 1.5) {\footnotesize{$A$}};
  \node at (1.5, 1.5) {\footnotesize{$B$}};
  \node at (2.5, 1.5) {\footnotesize{$-C$}};
  \node at (3.5, 1.5) {\footnotesize{$-D$}};
  \node at (0.5, 2.5) {\footnotesize{$0$}};
  \node at (1.5, 2.5) {\footnotesize{$A$}};
  \node at (2.5, 2.5) {\footnotesize{$0$}};
  \node at (3.5, 2.5) {\footnotesize{$0$}};
  \draw[color=black, line width=0.4mm] (2,0) rectangle (4,3);
  \node[fill=white,scale=0.5] at (2,3) {\Huge{$\oplus$}};
 \end{tikzpicture} \quad
  %%%%% Bild 2
  \begin{tikzpicture}[scale=0.8]
  \draw[help lines] (0,0) grid (4,3);
  \node at (0.5, 0.5) {\footnotesize{$B$}};
  \node at (1.5, 0.5) {\footnotesize{$C$}};
  \node at (2.5, 0.5) {\footnotesize{$-D$}};
  \node at (3.5, 0.5) {\footnotesize{$0$}};
  \node at (0.5, 1.5) {\footnotesize{$A$}};
  \node at (1.5, 1.5) {\footnotesize{$B$}};
  \node at (2.5, 1.5) {\footnotesize{$-C$}};
  \node at (3.5, 1.5) {\footnotesize{$-D$}};
  \node at (0.5, 2.5) {\footnotesize{$0$}};
  \node at (1.5, 2.5) {\footnotesize{$A$}};
  \node at (2.5, 2.5) {\footnotesize{$0$}};
  \node at (3.5, 2.5) {\footnotesize{$0$}};
  \draw[color=black, line width=0.4mm] (2,0) -- ++(1,0) -- ++(0,1) -- ++(-1,0) -- ++(0,2) -- ++(-1,0) -- ++(0,-2)
-- ++(1,0) --
++(0,-1);
\node[fill=white,scale=0.5] at (2,1) {\Huge{$\oplus$}};
 \end{tikzpicture} \quad
  %%%%% Bild 3
  \begin{tikzpicture}[scale=0.8]

  \draw[help lines] (0,0) grid (4,3);
  \node at (0.5, 0.5) {\footnotesize{$B$}};
  \node at (1.5, 0.5) {\footnotesize{$C$}};
  \node at (2.5, 0.5) {\footnotesize{$-D$}};
  \node at (3.5, 0.5) {\footnotesize{$0$}};
  \node at (0.5, 1.5) {\footnotesize{$A$}};
  \node at (1.5, 1.5) {\footnotesize{$B$}};
  \node at (2.5, 1.5) {\footnotesize{$-C$}};
  \node at (3.5, 1.5) {\footnotesize{$-D$}};
  \node at (0.5, 2.5) {\footnotesize{$0$}};
  \node at (1.5, 2.5) {\footnotesize{$A$}};
  \node at (2.5, 2.5) {\footnotesize{$0$}};
  \node at (3.5, 2.5) {\footnotesize{$0$}};
  \draw[color=black, line width=0.4mm] (0,0) rectangle (2,3);
  \node[fill=white,scale=0.5] at (2,0) {\Huge{$\oplus$}};
 \end{tikzpicture}

\end{center}
for $e_1$, $e_2$ and $e_3 \in \mathbb{R}^{3}$ respectively. Therefore, any matrix pencil $\mathcal{L}(\lambda)$ in
$\mathbb{DL}(P)$ with
ansatz vector $v \in \mathbb{R}^3$ can be expressed as $ \mathcal{L}(\lambda) = [ \, (v \otimes I_n) \; \mathcal{Z}
\, ]
\textnormal{Frob}_P(\lambda)$ with
$$ \mathcal{Z} = v_1 \begin{bmatrix} 0 & 0 \\ -C & -D \\ -D & 0 \end{bmatrix}
+ v_2 \begin{bmatrix} A & 0 \\ B & 0 \\ 0 & -D \end{bmatrix}
+ v_3 \begin{bmatrix} 0 & A \\ A & B \\ B & C \end{bmatrix}. $$
\end{example}

\section{Conclusions}
\label{sec:conclusions}
In this paper, we introduced a family of equations for matrix pencils that turn out to be a new source of
linearizations for
square and rectangular matrix polynomials $P(\lambda)$. We showed that these equations define vector spaces
$\mathbb{G}_{\eta +
1}(P)$ of matrix pencils in which almost every pencil is a strong linearization regardless whether $P(\lambda)$
is regular or singular. These spaces were named \enquote{block Kronecker ansatz spaces} since they contain the entire
family of
block Kronecker pencils as introduced in \cite{DopLPVD16} and share important properties with the \enquote{ansatz spaces}
from
\cite{MacMMM06}. We showed that the intersection of two block Kronecker ansatz spaces $\mathbb{DG}_{\eta + 1}(P) =
\mathbb{G}_{\eta
+ 1}(P) \cap \mathbb{G}_{k - \eta}(P)$ is never empty and contains a proper subspace $\mathbb{BG}_{\eta + 1}(P)$ of
block-symmetric matrix pencils. Still almost every pencil is a strong linearization in either $\mathbb{DG}_{\eta +
1}(P)$ and
$\mathbb{BG}_{\eta + 1}(P)$ given the case that zero is not an eigenvalue of $P(\lambda)$. Moreover, we presented a
simple
approach to the construction of matrix pencils in $\mathbb{DG}_{\eta + 1}(P)$ and $\mathbb{BG}_{\eta + 1}(P)$ and
showed that these spaces form nested sequences of vector spaces for increasing choices of $\eta$.
%We presented our results assuming the fields underlying our derivations are the real numbers $\mathbb{R}$.

Block Kronecker ansatz equations may be defined for other
polynomial bases as well (see, e.g., \cite{LawP16} for a clever generalization of block Kronecker pencils for the
Chebyshevbasis). Moreover, as we pointed out in Remark \ref{rem1}, the conceptual ideas presented here may even be formulated in the
abstract
framework of dual bases (i.e. \enquote{strong block minimal bases pencils}, see \cite{DopLPVD16} for more
information). A deeper
study in this direction is, at least to the authors
opinion, likely to give attractive novel results on how Fiedler pencils, block Kronecker pencils and ansatz spaces
interact.

\section{Acknowledgement}

Our sincere thanks go to both of the reviewers. Their helpful remarks and comments helped to improve the paper.
We gratefully appreciate this.

%\bibliographystyle{elsarticle-harv}
%\bibliography{references}
\end{document}